\documentclass[11pt,reqno]{amsart}
\usepackage{color, amsmath,amssymb, amsfonts, amstext,amsthm, latexsym}
\allowdisplaybreaks
\newcommand{\R}{{\mathbb R}}

\newcommand{\cF}{{\cal F}}
\newcommand{\RR}{{\mathbb R}}
\newcommand{\HH}{{\mathcal H}}
\newcommand{\EX}{{\mathbb E}}
\newcommand{\EE}{{\mathbb E}}

  \newcommand{\PX}{{\mathbb P}}
\newcommand{\PP}{{\mathbb P}}

\newcommand{\dn}{{ \nabla}}
 \newcommand{\di}{{\rm div\,}}

\newcommand{\s}{\sigma}
\newcommand{\e}{\varepsilon}

\newcommand{\om}{\omega}
\newcommand{\Om}{\Omega}
\newcommand{\D}{\Delta}
\newcommand{\de}{\delta}
\newcommand{\la}{\lambda}

\renewcommand{\k}{\kappa}
\renewcommand{\th}{\theta}

\renewcommand{\cF}{\mathcal F}
\newcommand{\cL}{{\mathcal L}}

\numberwithin{equation}{section}
\newtheorem{theorem}{Theorem}[section]
\newtheorem{defn}[theorem]{Definition}

\newtheorem{lemma}[theorem]{Lemma}
\newtheorem{remark}[theorem]{Remark}
\newtheorem{prop}[theorem]{Proposition}

\begin{document}
\title[ Stochastic 2D hydrodynamical systems]
{Stochastic 2D hydrodynamical type systems: \\ Well posedness
and large deviations}

\author[I. Chueshov and A. Millet]
{Igor Chueshov and Annie Millet  }

\address[I.~Chueshov]
{Department of  Mechanics and Mathematics\\
Kharkov National University\\
4~Svobody Square\\
61077, Kharkov, Ukraine} \email[I. Chueshov]{chueshov@univer.kharkov.ua}

\address[A.~Millet]
{ SAMOS-MATISSE, Centre d'\'Economie de la Sorbonne (UMR 8174), Universit\'{e} Paris 1,
 Centre Pierre Mend\`{e}s France,
90 rue de Tolbiac, F- 75634 Paris Cedex 13, France {\it and}
Laboratoire de Probabilit\'es et Mod\`eles Al\'eatoires (UMR 7599),
       Universit\'es Paris~6-Paris~7, Bo\^{\i}te Courrier 188,
          4 place Jussieu, 75252 Paris Cedex 05, France } \email[A.
~Millet]{amillet@univ-paris1.fr {\it and} annie.millet@upmc.fr}

\thanks{This research was partly supported
   by the research project BMF2003-01345 (A. Millet). }

\subjclass[2000]{Primary 60H15, 60F10; Secondary 76D06, 76M35. }

\keywords{Hydrodynamical models,   MHD, B\'{e}nard convection,
shell models of turbulence,
stochastic PDEs, large deviations}

\begin{abstract}
We deal with a class of abstract nonlinear stochastic models, which
covers many 2D hydrodynamical models including 2D Navier-Stokes
equations, 2D MHD  models and 2D magnetic B\'{e}nard problem and also
some shell models of turbulence. We
first prove the existence and uniqueness theorem for the class
considered. Our main result is a Wentzell-Freidlin type large
deviation principle for small multiplicative noise which we prove by
weak convergence method.
\end{abstract}

\maketitle


\section{Introduction}\label{s1}
In recent years there has been  a wide-spread interest in the study of
qualitative properties of stochastic models which describe
cooperative effects in fluids by taking into account macroscopic
parameters such as temperature or/and magnetic field. The
corresponding  mathematical models  consists in coupling  the stochastic Navier-Stokes
equations  with some transport or/and Maxwell equations,
which are also stochastically perturbed.
\par
Our goal in this paper is to suggest and develop a unified approach
which makes it possible to cover a wide class of mathematical coupled
models from fluid dynamics. Due to well-known reasons we  mainly
restrict ourselves to spatially two dimensional models. Our unified
approach is based on an abstract stochastic evolution equation in
some Hilbert space of the form
\begin{equation} \label{abstr-1}
\partial_t u  + \mathcal{A} u + \mathcal{B}(u, u) + {\mathcal R}(u)=  \sigma(t,u)\, \dot{W},
\end{equation}
where $\sigma(t,u)\, \dot{W}$ is a multiplicative  noise white in
time with spatial correlation. The hypotheses which we impose on the
linear operator $\mathcal{A}$, the bilinear mapping $\mathcal{B}$
and the operator ${\mathcal R}$
 are true in the case of 2D Navier-Stokes equation (where ${\mathcal R}=0$),
and also for some other classes of two dimensional hydrodynamical
models  such as magneto-hydrodynamic  equations, the Boussinesq
model for the  B\'{e}nard convection and 2D magnetic B\'{e}nard
problem. They also cover the case of  regular higher dimensional
problems such as the 3D Leray $\alpha$-model for the Navier-Stokes
equation  and some shell models
 of turbulence. See a further  discussion in Sect.\ref{s2.1} below.
\par
 For general abstract stochastic
evolution equation in infinite dimensional spaces we  refer to \cite{PZ92}.
However the hypotheses in \cite{PZ92} do not cover our hydrodynamical type model.
We also note the stochastic Navier-Stokes equations
were studied by many authors (see, e.g., \cite{CG94,FG95,MS02,VKF}
and the references therein).
\smallskip
\par
 We first  state the result on existence,   uniqueness and provide a priori estimates
  for a  weak (variational) solution
to the  abstract problem of the form \eqref{abstr-1} where the forcing term
also includes a stochastic
control term with a multiplicative coefficient
(see Theorem~\ref{th3.1}). As a particular case, we deduce well posedness
 when the Brownian motion
$W$ is translated by a random element of its Reproducing Kernel
Hilbert Space (RKHS), as well  a priori bounds of the solution with
constants which only depend on an a.s. bound of the RKHS norm of the
control. In all the concrete hydrodynamical examples described
above, the diffusion coefficient may contain a small multiple of the
gradient of the solution. Thus,  this result contains the
corresponding existence and uniqueness theorems and a priori bounds
for 2D Navier-Stokes equations (see, e.g. \cite{MS02,Sundar}), for
the Boussinesq model of  the  B\'{e}nard convection (see \cite{Ferrario}, \cite{DM}), and
also for the GOY shell model of turbulence (see \cite{BBBF06} and
\cite{MSS}). Theorem~\ref{th3.1} generalizes the existence result
for  MHD equations  given in \cite{BaDP} to the case of
multiplicative noise and also covers  new situations such as the 2D
magnetic B\'{e}nard problem, the 3D Leray $\alpha$-model
 and the Sabra shell model of turbulence.
\par
Our argument mainly follows the local monotonicity idea
suggested in \cite{MS02,Sundar}.
However, since we deal with an
abstract hydrodynamical model with a forcing term which contains a stochastic control
 under a  minimal set of hypotheses,
 the argument requires substantial modifications compared to that of
 \cite{Sundar} or \cite{MSS}.
It  relies on a two-step Gronwall lemma (see Lemma \ref{lemGronwall}
below and also \cite{DM}).
\smallskip
\par
Our main result  (see Theorem~\ref{PGDue}) is a Wentzell-Freidlin
type large deviation principle (LDP) for stochastic equations of the
form \eqref{abstr-1} with $\s:=\sqrt{\e}\s$ as $\e \to 0$, which
describes the exponential rate of convergence  of the  solution
$u:=u^\e$ to the deterministic solution  $u^0$. As in the classical
case of finite-dimensional diffusions, the rate function is
described by an energy minimization problem which involves
deterministic controlled equations. The LDP result is that which
would hold true if the solution were a continuous functional of the
noise $W$. Our proof   consists in transferring the LDP satisfied by
the Hilbert-valued Brownian motion $\sqrt{\e} W$ to that of a
Polish-space valued measurable functional of $\sqrt{\e} W$ as
established in \cite{BD00}; see also \cite{BD07}, \cite{DZ} and
\cite{DupEl97}. This is related to the Laplace principle. This
approach has been already  applied in several specific infinite
dimensional situations (see, e.g, \cite{Sundar} for 2D Navier-Stokes
equations, \cite{DM} for 2D B\'{e}nard convection, \cite{BD07} for
stochastic reaction-diffusion system,  \cite{WeiLiu,ReZh} for
stochastic $p$-Laplacian equation and some its generalizations,
 \cite{MSS} for the GOY shell model of turbulence).
 We also refer to  \cite{CR} for large
deviation results for evolution equations in the case on
non-Lipschitz coefficients.
\par
 Our result in Theorem~\ref{PGDue} comprehends a wide class of hydrodynamical systems.
In particular, in addition to the 2D Navier-Stokes equations and
the Boussinesq model mentioned above,
 Theorem~\ref{PGDue} also proves LDP for
2D MHD equations,
2D magnetic B\'{e}nard convection,
3D Leray $\alpha$-model, the Sabra shell model and dyadic model of turbulence.
Note that unlike \cite{Sundar} and \cite{MSS}, in order to give a complete argument
for the weak convergence (Proposition \ref{weakconv}) and the compactness result
(Proposition \ref{compact}),
 we need to prove a time approximation result (Lemma \ref{timeincrement}).
 This  requires to make stronger
assumptions on the diffusion coefficient $\sigma$, which should have some H\"older
time regularity, and  in the explicit hydrodynamical models,  no
longer can include the gradient of the solution (see also \cite{DM}).

\smallskip
\par
Note that the weak convergence approach has been used
 recently to prove LDP for stochastic evolution equations which
satisfy monotonicity and coercivity conditions by J.~Ren and
X.~Zhang \cite{ReZh} and   by  W.~Liu \cite{WeiLiu}. This class
of models does not contain the hydrodynamical
 systems considered in this paper and the main PDE model for this class is
 a reaction-diffusion equation
with a  nonlinear monotone diffusion term perturbed by globally
Lipschitz sub-critical nonlinearity. Let us point out one of the
main differences which explains why we have to impose some more time
regularity assumptions on the diffusion coefficient,  in contrast
with  \cite{ReZh}.  Indeed, unlike the situation considered in
\cite{ReZh} we do not assume the compactness of embeddings in the
corresponding Gelfand triple $V'\subset H\subset V$.
 Thus the elegant method in \cite{ReZh}, which is
based on some compactness property of the family of solutions in
${\mathcal C}([0,T],V')$  obtained by means of the Ascoli theorem,
cannot be applied here, and we  use a technical time
discretization.
Let us also point out that due to the bilinear term which arises in hydrodynamical
models,  the control of any moment of  the $V'$ norm  for  time increments of the
solution is not as good as that in \cite{ReZh}.
 On the other hand,  giving up   the compactness conditions allows
us to cover the important class of hydrodynamical models in {\em
unbounded} domains. The payoff for this is some H\"{o}lder condition in
time for the diffusion coefficient (see condition {\bf (C4)} below).  The
paper by W.~Liu \cite{WeiLiu}  does  not assume the  compactness embeddings
in the Gelfand triple
and  uses another  way to obtain the
compactness of the set of solutions in some set of time continuous
functions.
 This approach  is based on the rather strong
approximation hypothesis which involves some compact embeddings
(see Hypotheses (A4) and (A5) in \cite{WeiLiu}).
The technique introduced by W.~Liu   might be used in our
framework, in order to avoid the extra time regularity  of $\sigma$,
   at the expense of some
compact approximation condition of diffusion term.
 The  corresponding proof is quite involved and we do not adapt it to keep down the size
of the paper; we
 rather focus on the main technical problems raised
by our models,  even in the simple case of a time-independent diffusion coefficient.
  Note that even in the case of monotone and coercive equations,
the diffusion coefficient considered in \cite{ReZh} (resp. \cite{WeiLiu})
 cannot involve the gradient in order
to satisfy condition (H5) (resp.  (A4)).
\smallskip
\par
The paper is organized as follows. In Section \ref{s2} we describe
our mathematical model with details and provide the corresponding
motivations from the theory of (coupled) models of  fluid dynamics.
In this section we also formulate our abstract hypotheses
 and state the results about well posedness and apriori bounds
   of the abstract
stochastic equation which also may contain some random control term.
The proof of these  technical results is given  in the Appendix,
Section  \ref{App}. Note that these  preliminary results are  proved
in a more general framework than what is needed to establish the
large deviation principle. Indeed, we use them in \cite{CM-support}
where we characterize the support of the distribution of the
solution to the stochastic hydrodynamical equations. We formulate
and prove the  large deviations principle by the  weak convergence
approach in Section \ref{s4}. There we use  properties (such as a
priori bounds and localized time increment estimates) of this
stochastic control system as a preliminary step in order to apply
the  general LDP results from \cite{BD00,BD07} in our situation.

\section{Description of the model} \label{s2}
Let $(H, |.|)$ denote a separable Hilbert space, $A$ be an
(unbounded) self-adjoint positive linear operator on $H$. Set
$V=Dom(A^{\frac{1}{2}})$. For $v\in V$  set   $\|v\|=
|A^{\frac{1}{2}} v|$. Let $V'$ denote the dual of $V$ (with respect
to the inner product $(.,.)$ of $H$). Thus we have the Gelfand
triple $V\subset H\subset V'$. Let $\langle u,v\rangle $ denote the
duality between $u\in V$ and $v\in V'$ such that  $\langle
u,v\rangle =(u,v)$ for $u\in V$, $v\in H$,
 and let $B : V\times V \to
V'$ be a continuous mapping (satisfying the condition (\textbf{C1})
given below).
\par
The goal of this paper is to study stochastic perturbations of the following
abstract model in $H$
\begin{equation} \label{u0}
\partial_t u(t)  +   A u(t) + B\big(u(t),u(t) \big) + R u(t) =  f,
\end{equation}
where   $R$ is a linear bounded operator in $H$.
We assume that the mapping $B : V\times V \to
V'$  satisfies the following  antisymmetry and bound conditions:
\medskip\par
\noindent \textbf{Condition (C1):}{\it
\begin{itemize}
  \item  $B : V\times V \to V'$ is a  bilinear continuous  mapping.
  \item For $u_i\in V$, $i=1,2,3$,
\begin{equation} \label{as}
 \langle B(u_1, u_2)\, ,\, u_3\rangle  = - \,\langle  B(u_1, u_3)\, ,\, u_2\rangle.
\end{equation}
  \item
 There exists a Banach (interpolation) space
${\mathcal H} $ possessing the properties\\
 (i) $V\subset
{\mathcal H}\subset H;$\\
 (ii) there exists a constant $a_0>0$ such that
\begin{equation} \label{interpol}
\|v\|_\HH^2 \leq a_0 |v|\, \|v\|\quad \mbox{for any $v\in V$};
\end{equation}
(iii)
for every $\eta >0$ there exists $C_\eta >0$
 such that
\begin{align} \label{boundB}
| \langle B(u_1, u_2)\,, \, u_3\rangle | &\leq \eta\,  \|u_3\|^2 + C_\eta \, \|u_1\|_\HH^2 \,
 \|u_2\|_\HH^2, \quad for \;  u_i\in V, \; i=1,2,3.
\end{align}
\end{itemize}
}

\begin{remark}\label{re:1}
{\rm {\bf (1)}
The relation in \eqref{boundB} obviously implies that
\begin{align} \label{boundB-eq1}
| \langle B(u_1, u_2)\,, \, u_3\rangle | &\leq C_1  \|u_3\|^2 + C_2 \, \|u_1\|_\HH^2 \,
 \|u_2\|_\HH^2, \quad \mbox{\rm for }\; u_i\in V,\; i=1,2,3,
\end{align}
for some positive constants $C_1$ and $C_2$. On the other hand,
if we put in \eqref{boundB-eq1} $\eta C_1^{-1} u_3$ instead of $u_3$,
then we recover \eqref{boundB} with $C_\eta=C_1C_2 \eta^{-1}$
Thus the requirements \eqref{boundB} and\eqref{boundB-eq1} are equivalent.
If for $u_3\neq 0$ we put now $\eta =\|u_1\|_\HH\|u_2\|_\HH \|u_3\|^{-1}$ in
\eqref{boundB} with $C_\eta=C_1C_2 \eta^{-1}$, then using \eqref{as}
we obtain that for some constant $C>0$,
\begin{equation}\label{preB}
| \langle B(u_1, u_2)\,, \, u_3\rangle | \leq C  \, \|u_1\|_\HH \, \|u_2\| \,
 \|u_3\|_\HH, \quad \mbox{\rm for } \;  u_i\in V,\; i=1,2,3.
\end{equation}
It is also evident that \eqref{preB} and \eqref{as} imply \eqref{boundB}.
Thus the conditions in \eqref{boundB}, \eqref{boundB-eq1} and
\eqref{preB} are equivalent to each other.
\par
{\bf (2)}  To lighten notations for $u_1\in V$, set $B(u_1):=B(u_1,u_1)$;
relations \eqref{as}, \eqref{interpol} and 
\eqref{preB} yield  for every $\eta >0$ the existence of $C_\eta >0$
such that for $u_1, u_2\in V$,
\begin{equation} \label{boundB1}
| \langle B(u_1)\, , \, u_2\rangle | \leq \eta\,  \|u_1\|^2 + C_\eta \, |u_1|^2 \,
 \|u_2\|_\HH^4.
\end{equation}
Relations \eqref{as} and   \eqref{boundB1}   yield
\begin{equation}
|\langle B(u_1)-B(u_2)\, ,\, u_1-u_2\rangle| = |\langle B(u_1-u_2), u_2\rangle | \leq \eta \|u_1-u_2\|^2 +
C_\eta\,  |u_1 - u_2|^2\, \|u_2\|_\HH^4. \label{diffB1}
\end{equation}
}
\end{remark}
\medskip\par
\subsection{Motivation}
\label{s2.1}
The main motivation for the condition ({\bf C1}) is that it  covers a wide class
of 2D hydrodynamical models including the following ones.
An element of $\RR^2$ is denoted $u=(u^1,u^2)$.

\subsubsection{2D Navier-Stokes equation}
Let  $D $ be a bounded, open and simply connected domain  of $\RR^2$.
We consider the Navier-Stokes equation with
the Dirichlet (no-slip) boundary conditions:
\begin{equation} \label{1.1.0}
\partial_t u - \nu \Delta u + u\nabla u + \nabla p =f , \quad \mbox{\rm div}\, u=0
~~\mbox{ in }~~D,\qquad u=0\quad\mbox{on}\quad \partial D,
\end{equation}
where $u= (u^1(x,t), u^2(x,t))$ is the velocity of a fluid,   $p(x,t)$ is the pressure,
 $\nu$ the kinematic viscosity and $f(x,t)$ is an external  density of force per volume.
Let   $n$ denote
the outward normal to $\partial D$
and   let
\[
H_{(1)} = \{ f\in \left[L^2(D)\right]^2 : \di  f=0 \;  \mbox{\rm  in }\;  D \;
 \mbox{\rm  and } \; f\, .\, n=0 \; \mbox{\rm  on }\;  \partial D \}
\]
be endowed with the usual $L^2$ scalar product.
Here above we set
  $\di f= \sum_{i=1,2} \partial_i f_i$,
Projecting  on the space $H_{(1)}$ of divergence free vector fields,
problem \eqref{1.1.0}
 can be written in the form \eqref{u0} (with $R\equiv 0$) in the
space  $H_{(1)}$ (see e.g. \cite{Temam}), where
 $A$ is the Stokes operator defined by the  bilinear form
  \begin{equation}\label{stokes-f}
   a(u_1,u_2)=\nu \sum\limits_{j=1}^{2}
   \int_{D} \nabla u_1^j \cdot \nabla u_2^j \, dx,\quad
  \end{equation}
with $u_1, u_2\in
   V=V_1\equiv\left[H^1_0(D)\right]^2\cap H_{(1)}$.
The map  $B\equiv B_1 : V_1\times V_1 \to V_1'$ is defined by
 \begin{equation}\label{oper-B1}
 \langle B_1(u_1,u_2)\, ,\, u_3\rangle = \int_D  [u_1(x)  \dn u_2(x) ]\, u_3(x) dx
\equiv\sum_{i,j=1}^2 \int_D  u^j_1\: \partial_ju^i_2 \: u^i_3\,  dx,
 \quad u_i\in V_1.
 \end{equation}
Using integration by parts, Schwarz's and Young's inequality, one
checks that this map  $B_1$  satisfies the conditions of ({\bf C1})
with $\HH = \left[L^4(D)\right]^2\cap H_{(1)}$. The inequality in
\eqref{interpol} is the well-known Ladyzhenskaya inequality (see
e.g. \cite{Constantin} or \cite{Temam}).
\par
We can also include in \eqref{1.1.0} a Coriolis type force by changing
$f$ into $f-Ru$, where $R(u^1,u^2)= c_0(-u^2,u^1)$, for some constant $c_0$.
In this case we get \eqref{u0} with $R\neq 0$.
\par The case of unbounded domains $D$ (including $D=\RR^2$) can be also
considered in our abstract framework. For this we only need  to
shift the spectrum away from zero  by changing $A$ into $A+Id$ and
introducing $R=-Id$.

\subsubsection{2D magneto-hydrodynamic  equations}

We consider magneto-hydrodynamic (MHD)  equations
   for a viscous incompressible  resistive  fluid  in a 2D domain $D$,
which have the form
(see, e.g., \cite{morreau-mhd}):
  \begin{equation}\label{1.1u}
   \partial_tu-\nu_1\Delta u+
   u \nabla u= -\nabla \left(p+\frac{s}2 |b|^2\right)
   +s b \nabla b+
   f,
\end{equation}
\begin{equation}\label{1.1b}
   \partial_tb-\nu_2\Delta b+
   u \nabla b=
   b\nabla u+
   g,
\end{equation}
\begin{equation}\label{1.2}
   \di u=0, \quad  \di b =0 \quad
\end{equation}
where
$u=(u^1(x,t),u^2(x,t))$  and $b=(b^1(x,t),b^2(x,t))$
denote  velocity and magnetic fields,
$p(x,t)$ is a scalar pressure.
We consider the following
boundary conditions
\begin{equation}\label{1.1bc}
u=0, \quad   b\, .\, n=0, \quad \partial_1 b^2- \partial_2 b^1=0
\quad {\rm on}~~ \partial D
\end{equation}
In equations above  $\nu_1$ is the kinematic viscosity,
$\nu_2$ is the magnetic diffusivity (which is determined from
magnetic  permeability and conductivity of the fluid),
the positive parameter $s$ is defined by the relation  $s=Ha^2\nu_1\nu_2$,
where $Ha$ is the so-called Hartman number.
The given   functions  $f=f(x,t)$ and
 $g=g(x,t)$ represent  external volume
forces  and  the curl of external current applied to the fluid.
We refer to \cite{LadSol-mhd60}, \cite{DL-mhd72} and \cite{SeTe}
for the mathematical theory for the MHD equations.
\par
 Again, the above equations are a particular case of equation \eqref{u0}
for the following  spaces and operators which satisfy ({\bf C1}).
To see this we first note that without loss of generality we can assume that $s=1$
in \eqref{1.1u} (indeed, if $s\neq 1$ we can introduce a  new magnetic field
$b:=\sqrt{s}b$ and rescale the curl of the current $g:=\sqrt{s}g$).
 For the  velocity part of the  MHD equations,  we use the same spaces
 $H_{(1)}$ and  $V_1$ and the Stokes operator generated by the bilinear form defined by
\eqref{stokes-f}  with $\nu=\nu_1$. Now we denote this operator by $A_1$.
\par
As for the magnetic part we set  $H_{(2)}=H_{(1)}$ and
$V_2= \left[H^1(D)\right]^2\cap H_{(2)}$
and define another Stokes operator $A_2$ as an  unbounded operator in  $H_{(2)}$
generated by the bilinear form \eqref{stokes-f}  with $\nu=\nu_2$ when  considered on the space $V_2$.
\par
As in the previous case we can write \eqref{1.1u}--\eqref{1.1bc} in the form
\eqref{u0} in the space $H=H_{(1)}\times H_{(2)}$ with $A=A_1\times A_2$, $R\equiv 0$.
We also set $V=V_1\times V_2$ and define
 $B : V\times V \to V'$ by the relation
 \[
\langle B(z_1,z_2), z_3\rangle=\langle B_1(u_1,u_2), u_3\rangle
-\langle B_1(b_1,b_2), u_3\rangle+\langle B_1(u_1,b_2), b_3\rangle-
\langle B_1(b_1,u_2), b_3\rangle
\]
for $z_i=(u_i,b_i)\in V=V_1\times V_2$, where $B_1$ is given by
\eqref{oper-B1}. The conditions in ({\bf C1}) are satisfied with
$\HH=\big( \left[L^4(D)\right]^2\times \left[L^4(D)\right]^2 \big)
\cap H$.

\subsubsection{2D
Boussinesq model for the  B\'{e}nard convection.}

The next example is the following coupled system of Navier-Stokes and heat equations
from the  B\'{e}nard convection problem
(see e.g. \cite{Foias} and the references therein).
Let $D =(0, l) \times (0, 1)$ be a rectangular   domain  in the
vertical  plane,  $(e_1, e_2)$ the standard basis in
 $\R^2$  and $x=(x^1,x^2)$ an element of $\RR^2$.
 Denote by  $p(x,t)$ the pressure field,   $f, g$  external forces,
$ u=(u^1(x,t),u^2(x,t))$ the velocity field and  $\th=\th(x,t)$
the temperature field  satisfying the following system
\begin{eqnarray}
\partial_t u + u \nabla u-\nu \D u + \nabla p &=&  \th e_2  + f,
\quad \di u   = 0, \label{1.1}\\
\partial_t \th +u \nabla \th -u^2 -\k\D \th &=&
g,\label{eqn3}
\end{eqnarray}
with boundary conditions
\begin{eqnarray*}
u =0\;\;  \& \;\; \th=0 \;\; \mbox{on}\;\;x^2=0\; \mbox{and}\;x^2=1,   \\
u, p, \th, u_{x^1}, \th_{x^1} \; \mbox{are periodic in}\; x^1 \;
\mbox{with period}\; l.\footnotemark
\end{eqnarray*}
\footnotetext{Here and below this means that $\phi |_{x^1=0}=\phi |_{x^1=l}$
for the corresponding function.}

Here above  $\nu$ is the kinematic viscosity,
$\kappa$ is the thermal  diffusion coefficient.
Let
\begin{align*}
H_{(3)} =  & \left\{u\in \left[L^2(D)\right]^2,\; \di u=0, \;
u^2 |_{x^2=0}=u^2 |_{x^2=1}=0,\; u^1 |_{x^1=0}=u^1 |_{x^1=l}
  \right\}    \\
\end{align*}
and $H_{(4)}=    L^2(D)$.
We also denote
 \begin{align*}
V_3 =  & \left\{u\in H_{(3)}\cap \left[H^1(D)\right]^2,\;  u |_{x^2=0}=u |_{x^2=1}=0,\;
 u \; \mbox{is $l$-periodic in}\; x^1  \right\},    \\
V_4=  & \left\{\th \in H^1(D),\;  \; \th|_{x^2=0}=\th|_{x^2=1}=0,\;
\th \; \mbox{is $l$-periodic in}\; x^1
\right\}.
\end{align*}
Let $A_3$ be the Stokes operator in $H_{(3)}$
generated by the bilinear form \eqref{stokes-f}  considered on $V_3$
and $A_4$ be the operator in $H_{(4)}$
generated by the Dirichlet form
  \begin{equation*}
   a(\th_1,\th_2)=\kappa
   \int_{D} \nabla \theta_1 \cdot \nabla \th_2 \, dx,\quad
\th_1, \th_2\in  V_4.
  \end{equation*}
Again, the above equations are a particular case of equation \eqref{u0}
for the following spaces and operators which satisfy ({\bf C1}).
Let $H=H_{(3)}\times H_{(4)}$ and  $V=V_3\times V_4$. We
set $A(u,\theta)=(A_3 u \, ,\,  A_4 \th)$,
 $R(u,\theta)= - (\theta e_2\, ,\, u^2)$, and define the mapping
 $B : V\times V \to V'$ by the relation
 \[
\langle B(z_1,z_2), z_3\rangle=\langle B_1(u_1,u_2), u_3\rangle
+ \sum_{i=1,2} \int_D u_1^i \, \partial_i \,  \theta_2\;  \theta_3 \, dx
\]
for $z_i=(u_i,\th_i)\in V=V_3\times V_4$, where $B_1$ is given by
\eqref{oper-B1}. With these notations, the Boussinesq equations for
$(u,\theta)$ are a particular case of \eqref{u0} with condition
({\bf C1}) for
 $\HH=\big( \left[L^4(D)\right]^2\times L^4(D) \big)
\cap H$.

\subsubsection{2D magnetic B\'{e}nard problem.}
This is  the  Boussinesq model coupled with magnetic field (see \cite{GaPa}).
As above let $D =(0, l) \times (0, 1)$ be a rectangular   domain  in the
vertical plane,  $(e_1, e_2)$ the standard basis in
 $\R^2$. We consider the equations
\begin{eqnarray*}
\partial_t u + u \dn u-\nu_1 \D u + \nabla \left(p+\frac{s}2 |b|^2\right)
   -s b \nabla b &=&  \th e_2  + f,
\quad \di u   = 0, 
\\
\partial_t \th +u \dn \th -u^2 -\k\D \th &=&  \; f, 
\\
  \partial_tb-\nu_2\Delta b+
   u \nabla b -    b\nabla u &=&
   h, \quad \di b   = 0, 
\end{eqnarray*}
with boundary conditions
\begin{eqnarray*}
u =0\;\;  \& \;\; \th=0 \;\; \& \;\; b^2=0,\; \partial_2 b^1=0 \;\;  \mbox{on}\;\;x^2=0\; \mbox{and}\;x^2=1,   \\
u, p, \th, b, u_{x^1}, \th_{x^1}, b_{x^1} \; \mbox{are periodic in}\; x^1 \;
\mbox{with period}\; l.
\end{eqnarray*}
As for the MHD case we can assume that $s=1$. In this case we have \eqref{u0} for the
variable $z=(u,\th,b)$ with $H=H_{(3)}\times H_{(4)}\times H_{(5)}$,
where $H_{(3)}$ and $H_{(4)}$ are the same as in the previous example
and $H_{(5)}= H_{(3)}$. We also set
  $V=V_3\times V_4\times V_5$, where $V_3$ and $V_4$ are the same as above
  and $V_5= H_{(3)}\cap \left[H^1(D)\right]^2$. The operator $A$ is generated by the bilinear  form
\[
a(z_1,z_2)=\nu_1 \sum\limits_{j=1}^{2}
   \int_{D} \nabla u_1^j \cdot \nabla u_2^j \, dx+
\kappa
   \int_{D} \nabla \th_1 \cdot \nabla \th_2 \, dx
+\nu_2 \sum\limits_{j=1}^{2}
   \int_{D} \nabla b_1^j \cdot \nabla b_2^j \, dx\quad
   \]
for $z_i=(u_i,\th_i,b_i)\in V$. The bilinear operator $B$ is defined by
\begin{eqnarray*}
\langle B(z_1,z_2), z_3\rangle & = &\langle B_1(u_1,u_2), u_3\rangle
-\langle B_1(b_1,b_2), u_3\rangle
\\ & &
+\, \langle B_1(u_1,b_2), b_3\rangle-
\langle B_1(b_1,u_2), b_3\rangle
+ \sum_{i=1,2} \int_D u_1^i \, \partial_i \,  \theta_2\;  \theta_3 \, dx
\end{eqnarray*}
for $z_i=(u_i,\th_i,b_i)\in V$, where $B_1$ is given by
\eqref{oper-B1}. We also set $R(u,\theta, b)= - (\theta e_2\, ,\,
u^2, 0)$. It is easy to check that this model is an example of
equation \eqref{u0} with ({\bf C1}), where  $\HH=\big(
\left[L^4(D)\right]^2\times L^4(D)\times \left[L^4(D)\right]^2 \big)
\cap H$.

\subsubsection{3D Leray $\alpha$-model for Navier-Stokes equations}
The theory developed in this paper can be also applied to some 3D models.
As an example we consider 3D Leray $\alpha$-model
(see \cite{Ler34}; for recent development of this model we refer to \cite{titi2,titi1}  and
to the references therein).
In a bounded 3D domain $D$ we consider the following
equations:
\begin{align}
& \partial_t u - \nu \Delta u + v\nabla u + \nabla p =f ,\label{1.1.0-3da}
\\
& (1-\alpha \Delta)v=u,\quad \mbox{\rm div}\, u=0,\quad \mbox{\rm div}\, v=0
~~\mbox{ in }~~D, \label{1.1.0-3db} \\
&v=u=0\quad\mbox{on}\quad \partial D. \label{1.1.0-3dc}
\end{align}
where $u= (u^1, u^2, u^3)$ and $v= (v^1, v^2, v^3)$  are unknown fields,
 $p(x,t)$ is the pressure.
 In the space
\[
H = \{ u\in \left[L^2(D)\right]^3 : \di  u=0 \;  \mbox{\rm  in }\;  D \;
 \mbox{\rm  and } \; u\, .\, n=0 \; \mbox{\rm  on }\;  \partial D \}
\]
problem  \eqref{1.1.0-3da}--\eqref{1.1.0-3dc} can be written in the form
\[
u_t+Au+B(G_\alpha u, u)=\tilde f,
\]
where $A$ is the corresponding 3D Stokes operator  (defined as in
the  2D case by  the form $a(u_1,u_2)=\nu \sum_{j=1}^3 \int_D \nabla u_1^j\,
\nabla u_2^j\, dx$ on $V\equiv H\cap \left[H^1_0(D)\right]^3$),
$G_\alpha=\left(Id+\alpha\nu^{-1}A\right)^{-1}$ is the Green operator
and
 \begin{equation*}
 \langle B(u_1,u_2)\, ,\, u_3\rangle = \sum_{i,j=1}^3 \int_D  u^j_1\; \partial_ju^i_2\;  u^i_3\:  dx,
 \quad u_i\in V= H\cap \left[H^1_0(D)\right]^3.
 \end{equation*}
 Note that the embedding  $H^{1/2}(D)\subset L^3(D)$ implies that
the inequality \eqref{interpol} holds true for ${\mathcal H} =
\left[L^3(D)\right]^3 \cap H$. Furthermore, H\"older's inequality
and the embedding $H^1(D)\subset L^6(D)$ imply that for $u_1, u_2,
u_3\in V$,
\begin{eqnarray*}
| \langle B(G_\alpha u_1,u_2)\, ,\, u_3\rangle| &
\le & C\|u_2\| \, |G_\alpha u_1|_{L^6(D)} \, | u_3|_{L^3(D)} 
\le  C\|u_2\| \; \|G_\alpha u_1\|\;  | u_3|_{L^3(D)}  \\
& \le &
  C\|u_2\| \; | u_1|_{L^3(D)}\;  | u_3|_{L^3(D)},
\end{eqnarray*}
where the last inequality comes from the fact that $A^{\frac{1}{2}} G_\alpha$ is a bounded operator on
$H$, so that $\|G_\alpha u_1\|=  |A^{\frac{1}{2}} G_\alpha u_1| \leq C |u_1|\leq C |u_1|_{L^3(D)}$.
By Remark~\ref{re:1}(1)
this implies condition ({\bf C1}) for $B_\alpha( u_1,u_2):=B(G_\alpha u_1,u_2)$.

\subsubsection{Shell models of turbulence}
Let $H$ be a set of all sequences $u=(u_1, u_2,\ldots)$ of complex numbers
such that $\sum_n |u_n|^2<\infty$. We consider $H$ as a \emph{real} Hilbert space
endowed  with the inner product $(\cdot,\cdot)$ and the norm $|\cdot|$ of the form
\[
(u,v)={\rm Re}\,\sum_{n=1}^\infty u_n v_n^*,\quad
|u|^2 =\sum_{n=1}^\infty |u_n|^2,
\]
where $v_n^*$ denotes the complex conjugate of $v_n$. In this space $H$ we
consider the evolution equation \eqref{u0} with $R=0$ and with linear operator $A$
and bilinear mapping $B$ defined by the formulas
\[
(Au)_n =\nu k_n^2 u_u,\quad n=1,2,\ldots,\qquad
Dom(A)=\left\{ u\in H\, :\; \sum_{n=1}^\infty k_n^4 |u_n|^2<\infty\right\},
\]
where $\nu>0$, $k_n=k_0\mu^n$ with $k_0>0$ and $\mu>1$, and
\[
\left[B(u,v)\right]_n=-i\left( a k_{n+1} u_{n+1}^* v_{n+2}^*
+b k_{n} u_{n-1}^* v_{n+1}^* -a k_{n-1} u_{n-1}^* v_{n-2}^*
-b k_{n-1} u_{n-2}^* v_{n-1}^*
\right)
\]
for $n=1,2,\ldots$, where $a$ and $b$ are real numbers (here above we also assume that
$ u_{-1}= u_{0}=v_{-1}= v_{0}=0$). This choice  of $A$ and $B$ corresponds to the
so-called GOY-model (see, e.g., \cite{OY89}).
If we take
\[
\left[B(u,v)\right]_n=-i\left( a k_{n+1} u_{n+1}^* v_{n+2}
+b k_{n} u_{n-1}^* v_{n+1} +a k_{n-1} u_{n-1} v_{n-2}
+b k_{n-1} u_{n-2} v_{n-1}
\right),
\]
then we obtain the Sabra shell model introduced in \cite{LPPPV98}. In
 both cases the equation \eqref{u0} is an infinite sequence of ODEs.
\par
One can easily  show (see \cite{BBBF06} for the GOY model and \cite{CLT06}
for the Sabra model) that
the trilinear form
\[
\langle B(u,v), w\rangle\equiv {\rm Re}\, \sum_{n=1}^\infty [B(u,v)]_n\,  w_n^*
\]
possesses the property \eqref{as} and also satisfies the inequality
\[
\left|\langle B(u,v), w\rangle\right|\le C |u||A^{1/2} v| |w|,\quad
\forall  u,w\in H, \quad \forall v\in Dom(A^{1/2}).
\]
Thus by Remark~\ref{re:1}(1)
the condition ({\bf C1}) holds with $\HH= Dom(A^{s})$ for any choice
of $s\in [0,1/4]$.
\medskip\par
We can also consider the so-called dyadic model (see, e.g., \cite{KP05}
and the references therein) which can be written as an infinite system
of real ODEs of the form
\begin{equation}\label{dyadic-m}
\partial_t u_n+\nu\lambda^{2\alpha n} u_n-\lambda^{n} u^2_{n-1}+
\lambda^{n+1} u_n u_{n+1}=f_n, \quad n=1,2,\ldots,
\end{equation}
where  $\nu, \alpha>0$, $\lambda>1$, $u_0=0$.
Simple calculations show that under the condition $\alpha\ge 1/2$
the system \eqref{dyadic-m} can be written as \eqref{u0}
and that  condition ({\bf C1}) holds for
$[B(u,v)]_n=-\lambda^n u_{n-1} v_{n-1}+ \lambda^{n+1}\, u_n\, v_{n+1}$ and
$(Au)_n=\nu\, \lambda^{2\alpha n}\, u_n$.

\subsection{Stochastic model}
 We will consider a
stochastic   external random force   $f$
of the equation in \eqref{u0} driven by a Wiener process $W$
and whose intensity may depend on the solution $u$. More precisely,
let $Q$ be a linear
positive  operator in the Hilbert space $H$ which belongs to the trace class,
and hence is  compact. Let $H_0 = Q^{\frac12} H$. Then $H_0$ is a
Hilbert space with the scalar product
$$
(\phi, \psi)_0 = (Q^{-\frac12}\phi, Q^{-\frac12}\psi),\; \forall
\phi, \psi \in H_0,
$$
together with the induced norm $|\cdot|_0=\sqrt{(\cdot,
\cdot)_0}$. The embedding $i: H_0 \to  H$ is Hilbert-Schmidt and
hence compact, and moreover, $i \; i^* =Q$.
Let $L_Q\equiv L_Q(H_0,H) $ be the space of linear operators $S:H_0\mapsto H$ such that
$SQ^{\frac12}$ is a Hilbert-Schmidt operator  from $H$ to $H$. The norm in the space $L_Q$ is
  defined by  $|S|_{L_Q}^2 =tr (SQS^*)$,  where $S^*$ is the adjoint operator of
$S$. The $L_Q$-norm can be also written in the form
\begin{equation}\label{LQ-norm}
 |S|_{L_Q}^2=tr ([SQ^{1/2}][SQ^{1/2}]^*)=\sum_{k=1}^\infty |SQ^{1/2}\psi_k|^2=
 \sum_{k=1}^\infty |[SQ^{1/2}]^*\psi_k|^2
\end{equation}
for any orthonormal basis  $\{\psi_k\}$ in $H$.
\par
Let   $W(t)$ be a   Wiener process  defined   on a filtered
probability space $(\Om, \cF, \cF_t, \PX)$, taking values in $H$
and with covariance operator $Q$. This means that $W$ is Gaussian, has independent
time increments and that for $s,t\geq 0$, $f,g\in H$,
\[
\EE  (W(s),f)=0\quad\mbox{and}\quad
\EE  (W(s),f) (W(t),g) = \big(s\wedge t)\, (Qf,g).
\]
We also have the representation
\begin{equation}\label{W-n}
W(t)=\lim_{n\to\infty} W_n(t)\;\mbox{ in }\; L^2(\Om; H)\; \mbox{ with }
W_n(t)=\sum_{j=1}^n q^{1/2}_j \beta_j(t) e_j,
\end{equation}
where  $\beta_j$ are  standard (scalar) mutually independent Wiener processes,
$\{ e_j\}$ is an  orthonormal basis in $H$ consisting of eigen-elements of $Q$, with
$Qe_j=q_je_j$. For details concerning this Wiener process
we refer  to \cite{PZ92}, for instance.
\par
The noise intensity $\s: [0, T]\times V \to L_Q(H_0, H)$ of the stochastic perturbation
which we put in \eqref{u0} is
assumed to satisfy the following growth and Lipschitz conditions:
\medskip
\par
\noindent \textbf{Condition (C2):} {\it
$\s \in C\big([0, T] \times V; L_Q(H_0, H)\big)$,
 and there exist non negative  constants $K_i$ and $L_i$
such that for every $t\in [0,T]$ and $u,v\in V$:\\
{\bf (i)}   $|\s(t,u)|^2_{L_Q} \leq K_0+ K_1 |u|^2+ K_2 \|u\|^2$, \\
{\bf (ii)}   $|\s(t,u)-\s(t,v)|^2_{L_Q}
\leq L_1 |u-v|^2 + L_2 \|u-v\|^2$.
}
\smallskip

\begin{remark}\label{re:as-s}
 {\rm
Assume that
$\s \in C\big([0, T] \times Dom(A^s); L_Q(H_0, H)\big)$ for some $s<1/2$
is such that
\[
|\s(t,u)|^2_{L_Q} \leq K'_0+ K'_1 |A^su|^2,\quad
|\s(t,u)-\s(t,v)|^2_{L_Q}
\leq L' |A^s(u-v)|^2
\]
for every $t\in [0,T]$ and $u,v\in Dom(A^s)$
with some  positive constants $K'_0$, $K'_1$ and $L'$.
By interpolation we have that for some constant $c_0>0$ and any $\eta>0$ and $ u\in V$:
\begin{equation}\label{interpol-ineq}
|A^su|^2 \le  c_0 |A^{1/2}u|^{4s}|u|^{2-4s} 
\le
\eta |A^{1/2}u|^2 + C_\eta|u|^2 . 
\end{equation}
Therefore in this case
 the conditions in ({\bf C2}) are valid with
 positive constants $K_2$ and $L_2$ which  can be taken arbitrary  small.
 This observation  is important because in Theorem~\ref{th3.1} below we impose
 some restrictions on the values of the parameters $K_2$ and $L_2$.
}
\end{remark}

Recall that for $u\in V$, $B(u)=B(u,u)$. Consider the following stochastic  equation
\begin{equation} \label{u}
d u(t)  + \big[  A u(t) + B\big(u(t) \big) + R u(t) \big]\, dt
=  \sigma(t,u(t))\, dW(t).
\end{equation}
For technical reasons, in order to prove a large deviation principle for the
law the solution to \eqref{u}, we will need some precise estimates
on the solution of the equation  deduced from \eqref{u} by  shifting
 the Brownian $W$ by some
random element  (see e.g. \cite{Sundar} and \cite{DM}).
 This cannot be deduced from similar ones  on $u$ by means
of a Girsanov transformation since the Girsanov density is not
uniformly bounded when the intensity of the noise tends to zero
(see \cite{DM}). Thus we also need to consider the corresponding shifted problem.
\par
To describe a set of admissible random shifts  we introduce the class
 $\mathcal{A}$ as the  set of $H_0-$valued
$(\cF_t)-$predictable stochastic processes $h$ such that
$\int_0^T |h(s)|^2_0 ds < \infty, \; $ a.s.
Let
\[S_M=\Big\{h \in L^2(0, T; H_0): \int_0^T |h(s)|^2_0 ds \leq M\Big\}.\]
The set $S_M$ endowed with the following weak topology is a
  Polish space (complete separable metric space)
\cite{BD07}:
$ d_1(h, k)=\sum_{i=1}^{\infty} \frac1{2^i} \big|\int_0^T \big(h(s)-k(s),
\tilde{e}_i(s)\big)_0 ds \big|,$
where $
\{\tilde{e}_i(s)\}_{i=1}^{\infty}$ is an  orthonormal basis
for $L^2(0, T; H_0)$.
Define
\begin{equation} \label{AM}
 \mathcal{A}_M=\{h\in \mathcal{A}: h(\om) \in
 S_M, \; a.s.\}.
\end{equation}
\medskip

In order to define the stochastic control equation, we introduce another  intensity
coefficient $\tilde{\s}$ and also nonlinear feedback forcing $\tilde R$
(instead of $R$)
which satisfy
\medskip\par
\noindent \textbf{Condition (C3):} {\bf (i)}   {\it ${}\;{\tilde \s}
\in C\big([0, T] \times V; L(H_0, H)\big)$ and there exist constants
 $\tilde{K}_{\mathcal H}$,   $\tilde{K}_i$,  and $\tilde{L}_j$, for
$i=0,1$ and $j=1,2$ such that
\begin{align} |\tilde{\s}(t,u)|^2_{L(H_0,H)} \leq \tilde{K}_0 + \tilde{K}_1 |u|^2 +
\tilde{K}_\HH \|u\|_{\mathcal H}^2, & \quad \forall t\in [0,T], \;
 \forall u\in V, \label{tilde-s-b}\\
 |\tilde{\s}(t,u) -\tilde{\s}(t,v)  |^2_{L(H_0,H)} \leq \tilde{L}_1 |u-v|^2
+ \tilde{L}_2 \|u-v\|^2, & \quad  \forall t\in [0,T], \; \forall u,v\in V,
\label{tilde-s-lip}
\end{align}
where $|\cdot |_{L(H_0,H)}$ denotes the (operator) norm
in the space $L(H_0,H)$ of all bounded linear operators from $H_0$ into $H$.
}
\par\noindent
{\bf (ii)}
{\it
 $\tilde R : [0, T] \times H\mapsto H$ is a continuous mapping such that
\[
|\tilde R(t,0)|\le R_0,\quad |\tilde R(t,u) -\tilde R(t,v)|\le R_1 |u-v|,\quad
\forall u,v\in H,\; \forall t\in [0,T],
\]
for non-negative  constants $R_0$ and $R_1$.}
\begin{remark}\label{re:s-tilde-s}
{\rm In contrast with Condition ({\bf C2}) our hypotheses concerning the
control intensity coefficient $\tilde{\s}$ involve a weaker topology (we deal with
the operator norm $|\cdot |_{L(H_0,H)}$ instead of the trace class norm
$|\cdot |_{L_Q}$). However we require in \eqref{tilde-s-b} a stronger
bound (in the intermediate space $\HH$). One can see that any noise intensity $\s$
satisfies  Condition ({\bf C3})({\bf i}) provided Condition ({\bf C2})
holds with $K_2=0$.
}
\end{remark}
Let  $M >0$,  $h\in {\mathcal A}_M$ and  $\xi\in H$.
Under Conditions   ({\bf C2}) and ({\bf C3}) we consider the nonlinear
SPDE with initial condition
$u_h(0)=\xi$:
\begin{equation}  \label{uh}
d u_h(t)  + \big[  A u_h(t) + B\big(u_h(t) \big) + \tilde R (t, u_h(t)) \big]\, dt
 =  \sigma(t,u_h(t))\, dW(t) + \tilde{\s}(t, u_h(t)) h(t)\, dt.
\end{equation}
Fix $T>0$ and let $ X: = C\big([0, T]; H\big) \cap L^2\big(0, T;V\big) $
denote the Banach space with the norm defined by
\begin{equation} \label{norm}
 \|u\|_X = \Big\{\sup_{0\leq s\leq T}|u(s)|^2+ \int_0^T \|
u(s)\|^2 ds\Big\}^\frac12 .
\end{equation}
Recall that an  $(\cF_t)$-predictable stochastic process $u_h(t,\om)$ is called a
{\em  weak
solution } in $X$  for the stochastic equation \eqref{uh}  on $[0, T]$
 with initial condition $\xi$ if
$u\in X= C([0, T]; H) \cap L^2((0, T); V)$, a.s., and
satisfies
\begin{eqnarray*}
 (u_h(t), v)-(\xi, v) + \int_0^t[ \langle u_h(s), Av\rangle
 + \big\langle  B(u_h(s))\, ,\,  v\big\rangle + (\tilde R(s,u_h(s)),v)]ds \nonumber \\
 =   \int_0^t \big(\s(s,u_h(s)) dW(s),v \big) +
  \int_0^t \big( \tilde{\s}(s,u_h(s)) h(s) \, ,\, v\big)\, ds,\;\; {\rm a.s.},
\end{eqnarray*}
for all $v \in Dom(A)$ and all $t \in [0,T]$.
Note that this solution is a strong one in the probabilistic meaning, that is written
 in terms of stochastic integrals with respect to  the given Brownian motion $W$.
\par
The following assertion shows that   equation \eqref{u}, as well
as \eqref{uh}, has a unique solution in $X$, and the $X$-norm of the solution $u_h$ to \eqref{uh}
 satisfies a priori bounds  depending on $M$ when $h\in {\mathcal A}_M$.

\begin{theorem}\label{th3.1}
Assume that  Conditions {\bf (C1)}--{\bf (C2)} are satisfied, and that  either condition
(i) or else (ii) below hold true:
\par
(i) $\tilde{\sigma}=\sigma$ and $\tilde{R}$ satisfies condition {\bf (C3)(ii)};
\par
(ii) Condition {\bf (C3)} is satisfied.
\par
\noindent Then for every
 $M>0$ and $T>0$  there exists $\bar{K}_2:=K_2(T, M)>0$,
 (which also depends on $K_i$, $\tilde{K}_i$ and $R_i$,  $i=0,1$, and on
$\tilde{K}_\HH$)
 such that under  conditions  $\EX |\xi|^4 < \infty$,
 $h\in   \mathcal{A}_M$, $K_2\in [0, \bar {K}_2[$ and $L_2<2$
there exists a   weak solution $u_h$  in $X$ of the
equation  \eqref{uh} with
initial data  $u_h(0)=\xi \in H$.
 Furthermore, for this solution there exists a constant
 $C:=C(K_i, L_i,\tilde{K}_i  , \tilde{K}_{\mathcal H},
\tilde{L}_i, R_i, T,M)$
 such that
for $h\in {\mathcal A}_M$,
\begin{equation} \label{eq3.1}
 \EX\Big( \sup_{0\leq t\leq T}
 |u_h(t)|^4
+ \int_0^T \|u_h(t)\|^2\, dt+\int_0^T \|u_h(t)\|_\HH^4\, dt \Big) \leq C\, \big( 1+E|\xi|^4\big).
\end{equation}
 If the constant $L_2$ is small enough, the equation  \eqref{uh}  has a unique solution in $X$.
If one only requires $L_2<2$, then equation  \eqref{uh} has again
a unique solution in $X$ if   either
$\tilde{\sigma}=\sigma$  
 or if  the function $h$ possesses a deterministic bound,
i.e., there exists a (deterministic) scalar function $\psi(t)\in L^2(0,T)$
such that $|h(t)|_0\le \psi(t)$ a.s.
 \end{theorem}
%

The proof is similar to that given in \cite{DM}
for 2D Boussinesq model \eqref{1.1} and \eqref{eqn3}
(see also \cite{MS02,Sundar} for the case of 2D Navier-Stokes equations \eqref{1.1.0}).
 However, since we deal with an
abstract hydrodynamical model  under a  minimal set of hypotheses,
 the argument requires substantial modifications. For the sake of completeness
we give a detailed proof  in an   Appendix (see Section \ref{App}). Note that
only  the case $\tilde{\sigma}=\sigma$ will be used for the large deviations results.
However, the  general case is needed in some other frameworks,
such as the support characterization of the distribution
to equation \eqref{u} in \cite{CM-support}.


\section{Large deviations}  \label{s4}
  We consider large deviations using  a    weak convergence approach
  \cite{BD00, BD07}, based on variational representations of
  infinite dimensional    Wiener processes.
Let $\e>0$ and let $u^\e$ denote the solution to the following equation
\begin{equation} \label{ue}
du^\e(t) + [A u^\e(t) +B(u^\e(t)) +
\tilde R (t,u^\e(t))]\, dt =  \sqrt{\e}\,\sigma(t,u^\e(t))\, dW(t) \, , u^\e(0)=\xi\in H.
\end{equation}
\par
Theorem \ref{th3.1} shows that for a any choice of $K_2$ and $L_2$, for $\e$ small enough
the solution of \eqref{ue} exists and is unique in $X:=C([0,T],H)\cap L^2([0,T],V)$; it is denoted
 by $u^\e = {\mathcal G}^\e(\sqrt{\e} W)$
 for a  Borel measurable function ${\mathcal G}^\e: C([0, T], H) \to X$.
 A detailed proof of the  existence
of a such function ${\mathcal G}^\e$ is given in \cite{RoScZh}.
\par
Let $\mathcal{B}(X)$ denote the  Borel $\s-$field of the Polish space $X$ endowed with the metric
 associated with the norm defined by \eqref{norm}. We recall some classical   definitions;
by convention the infimum over an empty set is  $ +\infty$.
\begin{defn}
   The random family
$(u^\e )$ is said to satisfy a large deviation principle on
$X$  with the good rate function $I$ if the following conditions hold:\\
\indent \textbf{$I$ is a good rate function.} The function function $I: X \to [0, \infty]$ is
such that for each $M\in [0,\infty[$ the level set $\{\phi \in X: I(\phi) \leq M
\}$ is a    compact subset of $X$. \\
 For $A\in \mathcal{B}(X)$, set $I(A)=\inf_{u \in A} I(u)$.\\
\indent  \textbf{Large deviation upper bound.} For each closed subset
$F$ of $X$:
$$
\lim\sup_{\e\to 0}\; \e \log \PX(u^\e \in F) \leq -I(F).
$$
\indent  \textbf{Large deviation lower bound.} For each open subset $G$
of $X$:
$$
\lim\inf_{\e\to 0}\; \e \log \PX(u^\e \in G) \geq -I(G).
$$
\end{defn}

For all $h \in L^2([0, T], H_0)$, let $u_h$ be the solution of
the corresponding control equation (\ref{dcontrol}) with initial condition
$u_h(0)=\xi$:
\begin{eqnarray} \label{dcontrol}
d u_h(t) + [Au_h(t) +B(u_h(t))+\tilde R(t,u_h(t))]dt =\s(t,u_h(t)) h(t) dt .
\end{eqnarray}
Let ${\mathcal C}_0=\{ \int_0^. h(s)ds \, :\, h\in L^2([0,T], H_0)\}  \subset C([0, T], H_0)$.
Define ${\mathcal G}^0:   C([0, T], H_0)  \to X$ by
$ {\mathcal G}^0(g)=u_h $ for $  g=\int_0^. h(s)ds \in {\mathcal C}_0$
and ${\mathcal G}^0(g)=0$ otherwise.
Since  the argument below requires some information about the difference of the
solution at two different times, we need an additional assumption about the
regularity of the map $\sigma(.,u)$.
\smallskip
\par
\noindent{\bf Condition (C4)} ({\it Time H\"older regularity of $\sigma$}):
There exist  constants $\gamma>0$ and $C\geq 0$ such that for $t_1, t_2\in [0,T]$
and $u\in V$:
\[ |\sigma(t_1,u)-\sigma(t_2,u)|_{L_Q} \leq C \, \left( 1+ \| u\|\right) |t_1-t_2|^\gamma.\]
The following theorem is the main result of this section.
\begin{theorem}\label{PGDue}
Suppose  the conditions ({\bf C1}) and ({\bf C2}) with $K_2=L_2=0$ are satisfied.
Suppose furthermore that the conditions ({\bf C3 (ii)}) and ({\bf C4}) hold.
Then the solution  $(u^\e)$ to \eqref{ue} satisfies the large deviation principle in
$X=C([0, T]; H) \cap L^2((0, T); V)$,  with the good rate function
\begin{eqnarray} \label{ratefc}
 I_\xi (u)= \inf_{\{h \in L^2(0, T; H_0): \; u ={\mathcal G}^0(\int_0^. h(s)ds) \}}
 \Big\{\frac12 \int_0^T |h(s)|_0^2\,  ds \Big\}.
\end{eqnarray}
\end{theorem}
We at first prove   the following
technical lemma, which studies time increments of the solution to
a stochastic control problem  extending  both \eqref{ue} and \eqref{dcontrol}.
When  $\s$, $\tilde{\s}$ and $\tilde R$  satisfy ({\bf C2}) and  ({\bf C3}),
  $h\in   {\mathcal A}_M$,  the stochastic control problem is defined  as in \eqref{uh}:
    $ u_h^\e(0)=\xi$ and
\begin{equation} \label{uhe}
du_h^\e(t) + [A  u_h^\e(t) + B(u_h^\e(t)) + \tilde R (t,u^\e_h(t))]\, dt =
\sqrt{\e}\,  \sigma(t,u^\e_h(t))\, dW(t) +
\tilde{\sigma}(t,u_h^\e(t))\, h(t)\, dt.
\end{equation}
To state the lemma mentioned above,  we need the following notations.
For every integer $n$, let $\psi_n : [0,T]\to [0,T]$ denote a measurable map
such that for every $s\in [0,T]$,
$s\leq \psi_n(s) \leq \big(s+c2^{-n})\wedge T$ for some positive constant $c$.
 Given $N>0$,  $h\in {\mathcal A}_M$,
 and for  $t\in [0,T]$,  let
\[ G_N(t)=\Big\{ \omega \, :\, \Big (\sup_{0\leq s\leq t}  |u_h^\e(s)(\omega)|^2 \Big)\vee
\Big(  \int_0^t \|u_h^\e(s)(\omega)\|^2 ds \Big) \leq N\Big\}.\]
As in  Proposition~\ref{Galerkin},
 we can use a relaxed form of condition  ({\bf C3 (i)})    in order to perform
calculations in the following lemma; this more general setting is again used in \cite{CM-support}.
\begin{lemma} \label{timeincrement}
Let $\e_0, M,N>0$, $\sigma$
satisfy condition ({\bf C2})
and $\tilde{\sigma}$ satisfy  \eqref{tilde-s-lip} and the following condition
\eqref{tilde-s-bv}
\begin{align} |\tilde{\s}(t,u)|^2_{L(H_0,H)} \leq \tilde{K}_0 + \tilde{K}_1 |u|^2 +
\tilde{K}_2 \|u\|^2, & \quad \forall t\in [0,T], \;
 \forall u\in V, \label{tilde-s-bv}
\end{align}
instead of \eqref{tilde-s-b}.
 Assume that  $\xi\in L^4(\Om;H)$ and let $u_h(t)$
be solution  of \eqref{uhe}  satisfying the conclusion of Theorem~\ref{th3.1}.
 Then there exists a positive  constant
$C$ (depending on
$K_i,  \tilde{K}_i,i=0,1,2,  L_j, \tilde{L}_j, j=1,2, R_1, T, M, N, \e_0$)
 such that for
 any  $h\in {\mathcal A}_M$, $\e\in [0, \e_0]$,
\begin{equation}  \label{time}
I_n(h,\e):=\EX\Big[ 1_{G_N(T)}\; \int_0^T  |u_h^\e(s)-u_h^\e(\psi_n(s))|^2 \, ds\Big]
\leq C\, 2^{-\frac{n}{2}}.
\end{equation}
\end{lemma}
\begin{proof} The proof is  close to that of Lemma 4.2 in \cite{DM}.
However we deal with a class of more general functions $\psi_n(s)$ and do not
assume that $K_2=L_2= \tilde{K}_2=\tilde{L}_2=0$.
As above, to lighten the notation we skip
the time dependence of $\sigma$,  $\tilde{\s}$  and $\tilde{R}$.
Let $h\in {\mathcal A}_M$, $\e\geq 0$;
for any $s\in [0,T]$, It\^o's formula yields
\[
 |u_h^\e(\psi_n(s))-u_h^\e(s)|^2 =2\int_s^{\psi_n(s)} (u_h^\e(r)-u_h^\e(s), d u_h^\e(r))
+\e \int_s^{\psi_n(s)}|\s (u_h^\e(r))|^2_{L_Q}d r .
\]
 Therefore
$I_n(h,\e)=\sum_{1\leq i\leq 6} I_{n,i}$, where
\begin{eqnarray*}
I_{n,1}&=&2\, \sqrt{\e}\;  \EX\Big( 1_{G_N(T)} \int_0^T \!\! ds \int_s^{\psi_n(s)}\!  \big(
\sigma(u_h^\e(r)) dW(r) \, , \, u_h^\e(r)-u_h^\e(s) \big)\Big) , \\
I_{n,2}&=&{\e}\;  \EX \Big( 1_{G_N(T)} \int_0^T  \!\!ds \int_s^{\psi_n(s)} \!\!
|\sigma(u_h^\e(r))|_{L_Q}^2 \, dr\Big) , \\
I_{n,3}&=&2 \,  \EX \Big( 1_{G_N(T)} \int_0^T \!\! ds \int_s^{\psi_n(s)} \!\! \big(
\tilde{\sigma}(u_h^\e(r)) \, h(r)\,  , \, u_h^\e(r)-u_h^\e(s) \big)\, dr\Big), \\
I_{n,4}&=&- 2 \,  \EX \Big( 1_{G_N(T)} \int_0^T  \!\! ds \int_s^{\psi_n(s)} \!\!
 \big\langle A \, u_h^\e(r)\, , \, u_h^\e(r)-u_h^\e(s)\big\rangle  \, dr\Big) ,  \\
I_{n,5}&=&-2 \, \EX \Big( 1_{G_N(T)} \int_0^T  \!\! ds \int_s^{\psi_n(s)} \!\!
  \big\langle B( u_h^\e(r))\, , \,   u_h^\e(r)-u_h^\e(s)\big\rangle \, dr\Big) ,  \\
I_{n,6}&=&- 2  \, \EX \Big( 1_{G_N(T)} \int_0^T \!\!  ds \int_s^{\psi_n(s)} \!\!
  \big(\tilde R (u_{h}^\e(r))\, , \,  u_h^\e(r)-u_h^\e(s)\big)\, dr\Big) .
\end{eqnarray*}
Clearly  $G_N(T)\subset G_N(r)$ for $r\in [0,T]$.
In particular this means that  $|u_h^\e(r)|+|u_h^\e(s)|\le N$
 on $G_N(r)$ for $0\leq s\leq r\leq T$.
 We use this observation in the considerations
below.
\par
The
Burkholder-Davis-Gundy inequality and ({\bf C2}) yield for  $0\leq \e \leq \e_0$
\begin{eqnarray*}
|I_{n,1}|&\leq &
 6\sqrt{\e} \int_0^T ds \; \EX \Big( 
 \int_s^{\psi_n(s)} |\s(u_h^\e(r))|_{L_Q}^2
 1_{G_N(r)}\, | u_h^\e(r)- u_h^\e(s)|^2 \;  dr \Big)^{\frac{1}{2}} \\ 
&\leq & 6 \sqrt{2 \e_0 N }  \int_0^T ds \; \EX \Big( 
 \int_s^{\psi_n(s)}
[K_0+K_1\,  |u_h^\e(r)|^2  + K_2\,  \|u_h^\e(r)\|^2]\;   \, dr \Big)^{\frac{1}{2}}.
\end{eqnarray*}
Schwarz's inequality and Fubini's theorem as well as \eqref{eq3.1}, which holds
uniformly in  $\e\in ]0,\e_0]$ for fixed $\e_0>0$ since the constants $K_i$ and $L_i$
are multiplied by at most  $\e_0$,  imply
\begin{eqnarray} \label{In1}
|I_{n,1}|&\leq &
6 \sqrt{2 \e_0 N T }\, \Big[ \EX \int_0^T\!\! 
\big( K_0+K_1\, |u_h^\e(r)|^2 +
K_2\, \|u_h^\e(r)\|^2\big)\, \Big( \int_{(r-c2^{-n})\vee 0}^r ds\Big) \, dr
 \Big]^{\frac{1}{2}} \nonumber \\
&\leq & C_1 2^{-\frac{n}{2}}
\end{eqnarray}
for some constant $C_1$ depending only on $K_i$, $\tilde{K}_i, i=0,1,2$, $L_j$, $\tilde{L}_j, j=1,2$,
$R_1$,  $M$,
 $\e_0$, $N$ and $T$.
The property ({\bf C2}) and Fubini's theorem  imply that for $0\leq \e\leq \e_0$,
\begin{eqnarray} \label{In2}
|I_{n,2}|&\leq &\e \,   \EX \Big( 1_{G_N(T)}   \int_0^T \!\! ds \int_s^{\psi_n(s)} \!\!
\big(K_0+K_1|u_h^\e(r)|^2 + K_2 \|u_h^\e(r)\|^2 \big) dr\Big) \nonumber \\
&\leq &
\e_0  \EX \int_0^T  1_{G_N(T)} \, (K_0 +K_1\, N + K_2 \|u_h^\e(r)\|^2 )\, c 2^{-n}\, dr \;
\leq\;  C_2 2^{-n}
\end{eqnarray}
for some constant $C_2$ depending on $K_i$, $i=0,1,2$,  $\e_0$, $N$ and $T$.
Schwarz's inequality, Fubini's theorem,  ({\bf C2})  and the definition
\eqref{AM} of ${\mathcal A}_M$ yield
\begin{align}  \label{In3}
|I_{n,3}|&\leq 2   \;  \EX \Big( 1_{G_N(T)} \int_0^T \!\! ds \nonumber \\
& {}\quad \times \int_s^{\psi_n(s)} \!\!
 \big(\tilde{K}_0 +\tilde{K}_1|u_h^\e(r)|^2 + \tilde{K}_2 \|u_h^\e(r)\|^2
\big)^{\frac{1}{2}}\, |h(r)|_0 |\, u_h^\e(r)-u_h^\e(s)|\, dr\Big)
\nonumber \\
& \leq  4 \sqrt{N}  \; \EX  \int_{0}^T 1_{G_N(T)}
 |h(r)|_0 (\tilde{K}_0 +\tilde{K}_1 N  + \tilde{K}_2 \|u_h^\e(r)\|^2)^{\frac{1}{2}}
 \,
  \Big( \int_{(r-c2^{-n})\vee 0}^r ds \Big)\, dr  \nonumber \\
& \leq 4 \sqrt{N}\, c 2^{-n} \sqrt{M}\, \EX
  \Big( 1_{G_N(T)}\int_0^T (\tilde{K}_0 +\tilde{K}_1 N  + \tilde{K}_2 \|u_h^\e(r)\|^2)
 \, dr \Big)^{\frac{1}{2}}
\leq C_3\,  2^{-n},
\end{align}
for some constant $C_3$ depending on $\tilde{K}_i$, $i=0,1,2$, $M$ and $N$.
Using Schwarz's inequality we deduce that
\begin{eqnarray} \label{In4}
I_{n,4} &\leq &  2  \EX \Big( 1_{G_N(T)} \int_0^T  \!\! \! ds \int_s^{\psi_n(s)} \!\!\!
dr  \big[ -  \|u_h^\e(r)\|^2  +  \|u_h^\e(r)\| \|u_h^\e(s)\|\|\big]\Big) \nonumber \\
&\leq &
 \frac{1}{2}\;   \EX \Big( 1_{G_N(T)} \int_0^T ds \; \|u^\e_h(s)\|^2 \,
\int_s^{\psi_n(s)} dr \Big) \leq   c \;  N \;    2^{-(n+1)}.
\end{eqnarray}
The antisymmetry relation \eqref{as} and inequality \eqref{boundB1} yields
\begin{eqnarray*}
 \left|\big\langle B( u_h^\e(r))\, , \,   u_h^\e(r)-u_h^\e(s)\big\rangle\right|=
  \left|\big\langle B( u_h^\e(r))\, , \,  u_h^\e(s)\big\rangle\right|
\le  \|u_h^\e(r)\|^2 +C |u_h^\e(r)|^2 \|u_h^\e(s)\|^4_\HH .
\end{eqnarray*}
Therefore,
\begin{align} \label{In5}
|I_{n,5}| &\leq 2 \EX \Big( 1_{G_N(T)}\!\! \int_0^T  \!\!\! ds
\int_s^{\psi_n(s)} \!\!\! dr \ \|u^\e_h(r)\|^2 \Big)
\nonumber \\
& {}\quad +\, 2C \EX \Big( 1_{G_N(T)}\!\! \int_0^T  \!\!\! ds \|
u^\e_h(s)\|^4_\HH\int_s^{\psi_n(s)} \!\!\! dr  |u_h^\e(r)|^2 \Big)
\equiv I_{n,5}^{(1)}+  I_{n,5}^{(2)}.
\end{align}
Fubini's theorem implies
\begin{eqnarray} \label{In5a}
I_{n,5}^{(1)} &\leq& 2 \EX \Big( 1_{G_N(T)}\!\!
  \int_{0}^{T} \!\!\!
dr \ \|u^\e_h(r)\|^2  \int_{(r-c2^{-n})\vee 0}^r ds\Big) \nonumber \\
& \leq & 2c\, 2^{-n}\, \EX \Big( 1_{G_N(T)}\!\!
\int_0^T  \!\!\!
dr \ \|u^\e_h(r)\|^2 \Big) 
\le CN 2^{-n}.
\end{eqnarray}
Using  \eqref{interpol}, we deduce that  on $G_N(T)$ we have
\[
\int_0^T  \!\!\!  \| u^\e_h(s)\|^4_\HH ds\le a_0^2 \sup_{s\in [0,T]}| u^\e_h(s)\|^2
 \int_0^T  \!\!\! \| u^\e_h(s)\|^2 ds\le a_0^2  N^2.
\]
Thus
\begin{align} \label{In5b}
I_{n,5}^{(2)} &\leq 2C N
\EX \Big( 1_{G_N(T)}\!\! \int_0^T  \!\!\! ds
\| u^\e_h(s)\|^4_\HH\Big) c 2^{-n} \le 2 a_0^2  C N^3  c 2^{-n}.
\end{align}

Finally, Schwarz's inequality implies that
\begin{equation} \label{In6}
|I_{n,6}|\leq  2 \,\EX \Big[ 1_{G_N(T)} \int_0^T \!\! ds \int_s^{\psi_n(s)} \!\!\!
\left( R_0+R_1 |u_h^\e(r)|\right) \left( |u_h^\e(r)| + |u_h^\e(s)|\right)\,  dr\Big]
\leq
  C N^2 2^{-n}.
\end{equation}
Collecting the upper estimates from \eqref{In1}-\eqref{In6},
we conclude the proof of \eqref{time}.
\end{proof}

 In the setting of large deviations, we will use  Lemma \ref{timeincrement}
in the  case
when then $\s=\tilde\s$ satisfies  Condition ({\bf C2}) with $K_2=L_2=0$
and with the following choice of the function $\psi_n$.
 For any integer $n$  define a step function $s\mapsto \bar{s}_n$
on  $[0,T]$ by the formula
\begin{equation} \label{step-f}
\bar{s}_n= t_{k+1}\equiv(k+1)T 2^{-n}~~\mbox{ for }~~s\in [k T 2^{-n}, (k+1) T 2^{-n}[.
\end{equation}
Then the map $\psi_n(s)=\bar{s}_n$ clearly satisfies the previous requirements with $c=T$.
\par
Now we return to the setting of Theorem~\ref{PGDue}.
\par
Let $\e_0>0$,   $(h_\e , 0 < \e \leq \e_0)$  be a family of random elements
taking values in the set ${\mathcal A}_M$ given by \eqref{AM}.
Let $u_{h_\e}$, or  strictly speaking, $u^\e_{h_\e}$, be
the solution of the corresponding stochastic control equation
 with initial condition $u_{h_\e}(0)=\xi \in H$:
\begin{eqnarray} \label{scontrol}
d u_{h_\e} + [Au_{h_\e} +B(u_{h_\e})+\tilde{R}(t,u_{h_\e})]dt
=\s(t,u_{h_\e}) h_\e(t) dt+\sqrt{\e} \; \s(t,u_{h_\e}) dW(t) .
\end{eqnarray}
Note that for $W^\e_.=  W_. + \frac{1}{\sqrt \e}
 \int_0^. h_\e(s)ds$ we have that  $u_{h_\e}={\mathcal G}^\e\big(\sqrt{\e} W^\e
 \big)$.
 The following proposition
establishes the weak convergence of the family $(u_{h_\e})$ as
$\e\to 0$. Its proof is similar to that of Proposition 4.3 in
\cite{DM}, but allows time dependent coefficients $\tilde{R}$ and
$\s$.
\begin{prop}  \label{weakconv}
Suppose that the conditions ({\bf C1}) and ({\bf C2}) are satisfied
with $K_2=L_2=0$. Suppose furthermore that $\tilde{R}$ and $\s$
satisfy the conditions {\bf (C3)(ii)}  and {\bf (C4)}. Let $\xi $ be
${\mathcal F}_0$-measurable such that $E|\xi|_H^4<+\infty$, and let
$h_\e$ converge to $h$ in distribution as random elements taking
values in ${\mathcal A}_M$, where this set is defined by \eqref{AM}
and endowed with the weak topology of the space $L_2(0,T;H_0)$. Then
as $\e \to 0$, the solution $u_{h_\e}$ of \eqref{scontrol} converges
in distribution to the solution $u_h$ of \eqref{dcontrol} in
$X=C([0, T]; H) \cap L^2((0, T); V)$ endowed with the norm
(\ref{norm}). That is, as $\e \to 0$,
 ${\mathcal G}^\e\Big(\sqrt{\e} \big( W_. + \frac{1}{\sqrt{\e}} \int_0^. h_\e(s)ds\big) \Big)$  converges in
distribution to $ {\mathcal G}^0\big(\int_0^. h(s)ds\big)$ in $X$.
\end{prop}
\begin{proof}
 Since ${\mathcal A}_M$ is a Polish space (complete
separable metric space), by the Skorokhod representation theorem, we
can construct processes $(\tilde{h}_\e, \tilde{h}, \tilde{W}^\e)$
such that the joint distribution of $(\tilde{h}_\e,  \tilde{W}^\e)$
is the same as that of $(h_\e, W^\e)$,  the distribution of
$\tilde{h}$ coincides with that of $h$, and $ \tilde{h}_\e  \to
\tilde{h}$,  a.s., in the (weak) topology of $S_M$.   Hence a.s. for
every $t\in [0,T]$, $\int_0^t \tilde{h}_\e(s) ds - \int_0^t
\tilde{h}(s)ds \to 0$ weakly in $H_0$. To lighten notations, we will
write $(\tilde{h}_\e, \tilde{h}, \tilde{W}^\e)=(h_\e,h,W)$.

Let $U_\e=u_{h_\e}-u_h$; then  $U_\e(0)=0$ and
\begin{align}\label{difference2}
d U_\e + \big[AU_\e & +B(u_{h_\e})-B(u_h)+\tilde R(t,u_{h_\e})-\tilde R(t,u_h)\big]dt
\nonumber\\
& =\big[\s(t,u_{h_\e}) h_\e -\s(t,u_h) h\big] dt +\sqrt{\e} \;
\s(t,u_{h_\e}) dW(t).
\end{align}
 On any finite time interval $[0, t]$ with $t\leq T$, It\^o's formula,
\eqref{diffB1} with $\eta = \frac{1}{2}$  and condition ({\bf C2}) yield for $\e \geq 0 $:
\begin{align*}
&  |U_\e(t)|^2 +2 \int_0^t   \|U_\e(s)\|^2\, ds
 =   -2\int_0^t \big\langle B(u_{h_\e}(s))-B(u_{h}(s))\, ,\,  U_\e(s)\big\rangle ds\\
& \quad
-2 \int_0^t \big( \tilde R(s,u_{h_\e}(s))-\tilde R(s,u_h(s))\, ,\, U_\e(s)\big)\, ds
\\
& \quad +
 2\int_0^t \big(\s(s,u_{h_\e}(s))h_{\e}(s)-\s(s,u_h(s))\, h(s), U_\e(s) \big)ds \nonumber \\
& \quad +  2\sqrt{\e}\int_0^t \big( U_\e(s), \s(s,u_{h_\e}(s)) dW(s)\big)
+ \e\int_0^t|\s(s,u_{h_\e}(s))|^2_{L_Q}\, ds  \\
 &\;  \leq \int_0^t \!\: \|U_\e(s)\|^2\, ds +
  \sum_{i=1}^3 T_i(t,\e) + 2   \int_0^t\!\!
    \big( {C}_{\frac{1}{2}} \, \|u_h(s)\|^4_\HH +
R_1 + \sqrt{L_1}  |h_\e(s)|_0\big) |U_\e(s)|^2   ds,
\end{align*}
where
\begin{eqnarray*}
T_1(t,\e)&=& 2\sqrt{\e}\int_0^t \big( U_\e(s), \s(s,u_{h_\e}(s))\,  dW(s)  \big), \\
T_2(t,\e)&= & \e   \int_0^t (K_0+K_1|u_{h_\e}(s)|^2) ds , \\
T_3(t,\e)&=& 2\int_0^t \Big( \s(s,u_h(s))\, \big( h_{\e}(s)-h(s)\big), \, U_\e(s)\Big)\, ds.
\end{eqnarray*}
This yields the following inequality
\begin{equation} \label{total-error}
  |U_\e(t)|^2 +   \int_0^t\!\! \|U_\e(s)\|^2ds \leq
 \sum_{i=1}^3 T_i(t,\e)  + 2  \int_0^t\!\!
 \Big[  C_{\frac{1}{2}} \, \|u_h(s)\|_\HH^4 +
R_1 + \sqrt{L_1}  |h_\e(s)|_0\Big] |U_\e(s)|^2   \, ds.
\end{equation}
We want  to show that as $\e \to 0$,  $\|U_\e\|_X \to 0$ in
probability, which implies
 that $u_{h_\e} \to u_h$  in distribution in $X$.
 Fix $N>0$ and for $t\in [0,T]$ let
\begin{eqnarray*}  G_N(t)&=&\Big\{ \sup_{0\leq s\leq t} |u_h(s)|^2 \leq N\Big\} \cap
 \Big\{ \int_0^t \|u_h(s)\|^2  ds \leq N
\Big\}, \\
G_{N,\e}(t)&=&  G_N(t)\cap \Big\{ \sup_{0\leq s\leq t} |u_{h_\e}(s)|^2 \leq N\Big\} \cap
 \Big\{ \int_0^t \|u_{h_\e}(s)\|^2  ds \leq N
\Big\}.
\end{eqnarray*}
The proof consists in two steps.\\
\textbf{Step 1:}
For any $\e_0 \in ]0,1]0$, $ {\displaystyle \sup_{0<\e\leq \e_0}\; \sup_{h,h_\e \in {\mathcal A}_M}
\PX(G_{N,\e}(T)^c )\to 0 \;
\mbox{\rm as }\; N\to \infty.}$
\\
Indeed, for $\e\in ]0,\e_0]$, $h,h_\e \in {\mathcal A}_M$, the Markov inequality and
the a priori estimate  \eqref{eq3.1}, which holds uniformly in $\e\in ]0,\e_0]$, imply
\begin{align}\label{G-c}
&  \PX ( G_{N,\e}(T)^c )
\leq   \PX \Big(\sup_{0\leq s\leq T} |u_h(s)|^2 > N \Big)
 +\PX \Big(\sup_{0\leq
s\leq T} |u_{h_\e}(s)|^2 > N \Big)\nonumber \\
&\qquad \qquad\qquad   +  \PX\  \Big( \int_0^T \|u_h(s)\|^2 ds  >N\Big)
+ \PX\Big(\int_0^T \|u_{h_\e}(s)\|^2 \big)ds > N \Big)\nonumber \\
&\leq  \frac{1}{N}
  \sup_{ \;h, h_\e \in {\mathcal A}_M}
 \EX
\Big( \sup_{0\leq s\leq T} |u_h(s)|^2 + \sup_{0\leq s\leq T}
|u_{h_\e}(s)|^2
+ \int_0^T(\|u_h(s)\|^2+ \|u_{h_\e}(s)\|^2 )ds \Big) \nonumber \\
&\leq
 {C \, \big(1+ E|\xi|^4\big)}{N}^{-1},
\end{align}
for some constant $C$ depending on $T$ and $M$.

\noindent \textbf{Step 2:} Fix $N>0$,
 $h, h_\e \in {\mathcal A}_M$  such that  as $\e \to 0$, $h_\e\to h$ a.s. in the weak topology
 of $L^2(0,T;H_0)$; then  one has  as $\e \to 0$:
\begin{equation} \label{cv1}
\EX\Big[ 1_{G_{N,\e}(T)} \Big( \sup_{0\leq t\leq T } |U_\e(t)|^2 + \int_0^T  \|U_\e(t)\|^2 \, dt\Big)
\Big] \to 0.
\end{equation}
Indeed,  \eqref{total-error}  and Gronwall's lemma imply that on $G_{N,\e}(T)$,
\[ \sup_{0\leq t\leq T} |U_\e(t)|^2 \leq \Big[ \sup_{0\leq t\leq T}
 \big( T_1(t,\e) + T_3(t,\e)\big) +
\e C_*\Big] \, \exp\Big(2a_0 {C}_{\frac{1}{2}} N^2
 +  2R_1 T+2\sqrt{L_1 M T}\Big),\]
 where $C_*= T(K_0+K_1N)$.
We also use here the fact that by \eqref{interpol}
\[
\int_0^T  \!\!\!  \| u_h(s)\|^4_\HH ds\le a_0\sup_{s\in [0,T]}| u_h(s)|^2
 \int_0^T  \!\!\! \| u_h(s)\|^2 ds\le a_0 N^2\quad\mbox{on $G_{N,\e}(T)$.}
\]

Using again \eqref{total-error}  we deduce that for some constant
$\tilde{C}=C(T,M,N)$,
 one has for every $\e>0$:
\begin{equation} \label{*}
\EX\big( 1_{G_{N,\e}(T)} \, \|U_\e\|_X^2 \big) \leq \tilde{C} \Big( \e  + \EX \Big[ 1_{G_{N,\e}(T)} \,
\sup_{0\leq t\leq T} \big( T_1(t,\e) + T_3(t,\e)\big) \Big] \Big).
\end{equation}
Since the sets $G_{N,\e}(.)$ decrease, $\EX\big(1_{G_{N,\e}(T)}
 \sup_{0\leq t\leq T} |T_1(t,\e)|\big) \leq
\EX(\lambda_\e)$, where
\[ \lambda_\e := 2\sqrt{\e} \sup_{0\leq t\leq T }
\Big|\int_0^t 1_{G_{N,\e}(s)} \Big( U_\e(s), \s(s,u_{h_\e}(s)) dW(s) \Big)\Big|.\]
The scalar-valued random variables $\lambda_\e$ converge to 0 in $L^1$ as $\e\to 0$.
Indeed, by the Burkholder-Davis-Gundy inequality, ({\bf C2}) and the definition of $G_{N,\e}(s)$,
we have
\begin{align}
\EX (\lambda_\e) & \leq  6\sqrt{\e} \; \EX
\Big\{\int_0^T 1_{G_{N,\e}(s)} \,  |U_\e(s)|^2 \;
 |\s(s, u_{h_\e}(s))|^2_{L_{Q}}
 ds\Big\}^\frac12 \nonumber \\
&\leq  6\sqrt{\e} \; \EX \Big[  \Big\{ 4N \int_0^T 1_{G_{N,\e}(s)}\,
 (K_0+ K_1| u_{h_\e}(s) |^2 )  ds\Big\}^\frac12 \Big]
\leq  C(T,N) \,   \sqrt{\e}. \label{lambda1}
\end{align}
\par
 In further estimates we use Lemma~\ref{timeincrement} with $\psi_n=\bar{s}_n$, where
$\bar{s}_n$ is the step function defined in \eqref{step-f}.
For any $n,N \geq 1$, if we set $t_k=kT2^{-n}$ for $0\leq k\leq 2^n$,  we obviously have
\begin{equation}\label{t3}
\EX\Big( 1_{G_{N,\e}(T)}\sup_{0\leq t\leq T} |T_3(t,\e)| \Big)
\leq 2\;  \sum_{i=1}^4  \tilde{T}_i(N,n, \e)+ 2 \; \EX \big( \bar{T}_5(N,n,\e)\big),
\end{equation}
 where
\begin{align*}
\tilde{T}_1(N,n,\e)=& \EX \Big[ 1_{G_{N,\e}(T)} \sup_{0\leq t\leq T}  \Big| \int_0^t  \Big(  \s(s,u_h(s))
  \big(h_\e(s)-h(s)\big) \, ,\,
\big[U_\e(s)-U_\e(\bar{s}_n)\big] \Big)  ds\Big| \Big] ,\\
\tilde{T}_2 (N,n,\e)=& \EX\Big[ 1_{G_{N,\e}(T)} \\ & {}\quad
\times\sup_{0\leq t\leq T} \Big| \int_0^t
\Big( [\s(s,u_h(s)) - \s(\bar{s}_n, u_h(s))] (h_\e(s)-h(s))\, ,\, U_\e(\bar{s}_n)\Big)ds \Big|\Big], \\
\tilde{T}_3 (N,n,\e)=& \EX\Big[  1_{G_{N,\e}(T)}
\\ & {}\quad \times
\sup_{0\leq t\leq T} \Big| \int_0^t \Big( \big[ \s(\bar{s}_n,u_h(s))
- \s(\bar{s}_n, u_h(\bar{s}_n))\big]
\big(h_\e(s) - h(s) \big)\, ,\, U_\e(\bar{s}_n)\Big) ds\Big| \Big] ,\\
\tilde{T}_4(N,n,\e)=& \EX \Big[  1_{G_{N,\e}(T)} \sup_{1\leq k\leq 2^n}  \sup_{t_{k-1}\leq t\leq t_k}
\Big| \Big( \sigma(t_k, u_h(t_k))  \int_{t_{k-1}}^t (h_\e(s)-h(s))\, ds \, ,\, U_\e(t_k)\Big) \Big| \Big]\\
\bar{T}_5(N,n, \e)=&  1_{G_{N,\e}(T)} \sum_{k=1}^{2^n} \Big| \Big( \s(t_k, u_h(t_k))
\int_{t_{k-1}}^{t_k }  \big(h_\e(s)-h(s)\big)\, ds \, ,\,
U_\e(t_k )\Big) \Big| .
\end{align*}
Using Schwarz's inequality, ({\bf C2}) and Lemma~\ref{timeincrement}
with $\psi_n=\bar{s}_n$,
we deduce that  for some constant
$\bar{C}_1:= C(T,M,N)$ and any   $\e \in ]0, \e_0]$,
\begin{align} \label{eqT1}
&  \tilde{T}_1(N,n,\e)\leq
 \EX\Big[ 1_{G_{N,\e}(T)}  \int_0^T
 \big( K_0+K_1|u_h(s)|^2\big)^{\frac{1}{2}}
|h_\e(s)-h(s)|_0\, \big| U_\e(s)-U_\e(\bar{s}_n)\big|\, ds\Big]  \nonumber \\
& \quad   \leq
\Big(  \EX \Big[ 1_{G_{N,\e}(T)}  \int_0^T  \big\{ |u_{h_\e}(s) - u_{h_\e}(\bar{s}_n)|^2 +
 |u_{h}(s) - u_{h}(\bar{s}_n)|^2 \big\}\, ds\Big] \Big)^{\frac{1}{2}}
\nonumber \\
&\qquad \times
\sqrt{2(K_0 + K_1N)}\;  \Big( \EX  \int_0^T |h_\e(s)-h(s)|_0^2\, ds \Big)^{\frac{1}{2}}
\leq \bar{C}_1 \;  2^{-\frac{n}{4}}.
\end{align}
A similar computation based on ({\bf C2})  and Lemma \ref{timeincrement} yields  for
some constant $\bar{C}_3:=C(T,M,N)$ and any   $\e \in ]0, \e_0]$
\begin{align} \label{eqT2}
 \tilde{ T}_3 (N,n,\e) &\leq  \sqrt{2NL_1} \Big( \EX \Big[ 1_{G_{N,\e}(T)}  \int_0^T\!\!
 |u_{h}(s) - u_{h}(\bar{s}_n)|^2 \, ds\Big]  \Big)^{\frac{1}{2}} \Big(
 \EX \int_0^T \!\! |h_\e(s)-h(s)|_0^2  ds \Big)^{\frac{1}{2}}
 \nonumber \\
& \leq \bar{C}_3 \;  2^{-\frac{n}{4}}.
\end{align}
The H\"older regularity {\bf (C4)} imposed on $\s(.,u)$ and Schwarz's inequality imply that
\begin{equation}  \label{eqHolder}
\tilde{T}_2(N,n,\e)\leq C \, \sqrt{N} \, 2^{-n\gamma}\, \EX\Big(
1_{G_{N,\e}(T)} \int_0^T\left(1+\| u_h(s)\|\right)  |h_\e(s)-h(s)|\, ds \Big)
\leq \bar{C}_2  2^{-n\gamma}
\end{equation}
for some constant $\bar{C}_2=C(T,M,N)$.
Using Schwarz's inequality and ({\bf C2}) we deduce  for $\bar{C}_4=C(N,M)$
and any $\e \in ]0, \e_0]$
\begin{align} \label{eqT3}
\tilde{T}_4(N,n,\e)&\leq  \EX \Big[  1_{G_{N,\e}(T)} \sup_{1\leq k\leq 2^n}
\big(K_0+K_1| u_h(t_k)|^2
\big)^{\frac{1}{2}} \int_{t_{k-1}}^{t_k}\!\! |h_\e(s)-h(s)|_0
\, ds \, |U_\e(t_k)| \Big]\nonumber \\
&\leq 2 \sqrt{ N(K_0+K_1N)}\;  \EX\Big( \sup_{1\leq k\leq 2^n}
 \int_{t_{k-1}}^{t_k} |h_\e(s)-h(s)|_0 \, ds\Big)
\leq 4  \bar{C}_4\;  2^{-\frac{n}{2}}.
\end{align}
Finally, note that the weak convergence of $h_\e$ to $h$ implies that for any $a,b\in [0,T]$, $a<b$,
the integral $\int_a^b h_\e(s) ds \to \int_a^b h(s) ds$ in the weak topology of $H_0$.
Therefore, since for
the operator $\sigma(t_k, u_h(t_k))$ is compact from $H_0$ to $H$, we deduce that
for every $k$,
\[
\Big| \sigma(t_k, u_h(t_k)) \Big(   \int_{t_{k-1}}^{t_k} h_\e(s) ds - \int_{t_{k-1}}^{t_k} h(s) ds
  \Big) \Big|_H \to 0~~\mbox{ as }~~\e \to 0.
\]
Hence a.s. for fixed $n$ as $\e \to 0$, $\bar{T}_5 (N,n,\e,\omega) \to 0$. Furthermore,
 $\bar{T}_5(N,n,\e,\omega)
\leq C(K_0,K_1,N, M)$ and hence the dominated convergence theorem proves that for any
fixed $n,N$, $\EX(\bar{T}_5(N,n,\e))\to 0$ as $\e \to 0$.
\par
Thus, \eqref{t3}--\eqref{eqT3} imply that for any fixed $N\geq 1$  and any integer $n\geq 1$
\begin{equation*}
\limsup_{\e\to 0}\EX \Big[ 1_{G_{N,\e}(T)} \sup_{0\leq t\leq T} |T_3(t,\e)|\Big] \leq
C_{N,T,M}\;  2^{-n(\gamma \wedge \frac{1}{4} )}.
\end{equation*}
Since $n$ is arbitrary, this yields for any integer $N\geq 1$:
\begin{equation*}
\lim_{\e\to 0}\EX \Big[ 1_{G_{N,\e}(T)} \sup_{0\leq t\leq T} |T_3(t,\e)|\Big] =0 .
\end{equation*}
Therefore from \eqref{*} and \eqref{lambda1} we obtain \eqref{cv1}.
By the Markov inequality
\[
\PP(\|U_\e\|_X  > \de )  \leq  \PP(G_{N,\e}(T)^c )+ \frac{1}{\delta^2}
 \EX\Big( 1_{G_{N,\e}(T)} \|U_\e\|^2_X\Big)
~~\mbox{ for any }~~ \de>0.
\]
Finally,  \eqref{G-c} and \eqref{cv1} yield that for any integer $N\geq 1$,
\[
\limsup_{\e\to 0} \PP(\|U_\e\|_X  > \de )\le  C(T,M) N^{-1},
\]
for some constant $C(T,M)$ which does not depend on $N$.
This implies  $\lim_{\e\to 0} \PP(\|U_\e\|_X  > \de )=0$ for any $\delta>0$,
which  concludes the proof of the proposition.
\end{proof}

The following compactness result
is the second ingredient which allows to transfer the  LDP  from $\sqrt{\e} W$ to $u^\e$.
 Its proof is similar to that of Proposition \ref{weakconv}
and easier; it will be sketched (see also \cite{DM}, Proposition 4.4).
\begin{prop}  \label{compact}
Suppose that ({\bf C1}) and ({\bf C2}) hold  with $K_2=L_2=0$  and that conditions {\bf (C3)(ii)}
 and {\bf (C4)} hold. Fix  $M>0$, $\xi\in H$ and let
$ K_M=\{u_h \in X :  h \in S_M \}$,
where $u_h$ is the unique solution of the deterministic control
equation \eqref{dcontrol} and $X  = C([0, T]; H) \cap L^2(0, T; V)$.
Then $K_M$ is a compact subset of  $ X$.
\end{prop}
\begin{proof}
Let $\{u_n\}$ be a sequence in $K_M$, corresponding to solutions of
(\ref{dcontrol}) with controls $\{h_n\}$ in $S_M$:
\begin{eqnarray*} 
d u_n(t) + \big[A u_n(t) +B(u_n(t))+\tilde R (t,u_n(t))\big]dt =\s(t,u_n(t)) h_n(t) dt, \;\;
u_n(0)=\xi.
\end{eqnarray*}
 Since $S_M$ is a
bounded closed subset in the Hilbert space $L^2(0, T; H_0)$, it
is weakly compact. So there exists a subsequence of $\{h_n\}$, still
denoted as $\{h_n\}$, which converges weakly to a limit $h$ in
$L^2(0, T; H_0)$. Note that in fact $h \in S_M$ as $S_M$ is
closed. We now show that the corresponding subsequence of
solutions, still denoted as  $\{u_n\}$, converges in $ X$
 to $u$ which is the solution of the
following ``limit'' equation
\begin{eqnarray*}
d u(t) + [A u(t) +B(u(t))+\tilde R(t,u(t))]dt =\s(t, u(t)) h(t) dt, \;\;
u(0)=\xi.
\end{eqnarray*}
This will complete the proof of the compactness of $K_M$.
  To ease notation
 we will often drop the time parameters $s$, $t$, ... in the equations and integrals.
\par
Let $U_n=u_n-u$; using \eqref{diffB1} with $\eta=\frac{1}{2}$, ({\bf C2})  and Young's inequality,
we deduce that  for $t\in [0, T]$,
\begin{align}
& |U_n(t)|^2 +2 \int_0^t\!\!  \|U_n(s)\|^2  ds = \nonumber \\ &
\quad -2\int_0^t \!\! \big\langle B(u_n(s))-B(u(s)) , U_n(s)\big\rangle\, ds
 - 2\int_0^t \big( \tilde R(s,u_{n}(s))-\tilde R(s,u_h(s)) , U_n(s)\big) ds \nonumber \\
&\quad   + 2\int_0^t \Big\{ \Big( \big[ \s(s,u_n(s))
-\s(s,u(s))\big] h_{n}(s), U_n(s)\Big) \nonumber \\ &{}\qquad
{}\qquad{}\qquad{}\qquad {}\qquad{}\qquad{}\qquad
 + \big( \s(s,u(s)) \big(h_n(s)-h(s)\big)\, ,\, U_n(s)\big)\Big\} ds  \nonumber \\
& \leq   \int_0^t\!\!  \|U_n(s)\|^2  ds + 2 \int_0^t |U_n(s)|^2 \big(
C_{\frac{1}{2}}\|u(s)\|^4_\HH
 + R_1 +\sqrt{ L_1} \, |h_n(s)|_0\big)\, ds
\nonumber \\
&\quad
 + 2 \int_0^t \Big(  \s(s,u(s))\, [h_{n}(s)-h(s)]\; ,\; U_n(s)\Big) \, ds  .
\label{error1}
\end{align}
The inequality \eqref{eq3.1} implies that there exists a finite positive constant $\bar{C}$ such that
\begin{equation}\label{est-uni}
  \sup_n \Big[ \sup_{0\leq t\leq T} \big(|u(t)|^2 + |u_n(t)|^2\big) + \int_0^T
\big( \|u(s)\|^2 +\|u(s)\|^4_\HH+\|u_n(s)\|^2\big)ds \Big] =  \bar{C}.
\end{equation}
Thus Gronwall's lemma implies that
\begin{equation} \label{errorbound}
 \sup_{ t\leq T} |U_n(t)|^2 +\int_0^T \|U_n(t)\|^2\, dt  \leq
\exp\Big(2
  \big( {C}_{\frac{1}{2}} \bar{C}  + R_1\,  T + \sqrt{L_1 M T}  \big)\Big)\,
  \sum_{i=1}^5 I_{n,N}^{i}  ,
\end{equation}
where,  as in the proof of  Proposition~\ref{weakconv}, we have:
\begin{eqnarray*}
I_{n,N}^1&=& \int_0^T \big| \big( \s(s,u(s))\,  [h_n(s)- h(s)]\, ,\, U_n(s)-U_n(\bar{s}_N)\big)\big|\, ds,
\\
I_{n,N}^2&=& \int_0^T\Big|  \Big( \big[ \s(s,u(s)) - \s(\bar{s}_N, u(s)) \big]  [h_n(s)-h(s)]\, ,\,
U_n(\bar{s}_N)\Big)\Big| \, ds,  \\
I_{n,N}^3&=& \int_0^T\Big|  \Big( \big[ \s(\bar{s}_N,u(s)) - \s(\bar{s}_N,u(\bar{s}_N))\big]
  [h_n(s)-h(s)]\, ,\,
U_n(\bar{s}_N)\Big)\Big| \, ds,  \\
I_{n,N}^4&=& \sup_{1\leq k\leq 2^N} \sup_{t_{k-1}\leq t\leq t_k} \Big|\Big( \s(t_k,u(t_k ))
 \int_{t_{k-1 }}^t (h_\e(s)-h(s))
ds \; ,\; U_n(t_k) \Big)\Big| , \\
I_{n,N}^5&=& \sum_{k=1}^{2^N} \Big( \s(t_k, u(t_k )) \,  \int_{t_{k-1}}^{t_k }
[ h_n(s)-h(s)]\, ds   \; ,\; U_n(t_k ) \Big) .
\end{eqnarray*}
Schwarz's inequality, ({\bf C2}) and Lemma \ref{timeincrement} imply that
for some constants   $C_i$,  which do not
depend on $n$ and $N$,
\begin{align} \label{estim1}
I_{n,N}^1 &\leq C_0 \Big( \int_0^T\!\! |h_n(s)-h(s)|_0^2 ds \Big)^{\frac{1}{2}}
\Big( \int_0^T \!\! \big( |u_n(s)-u_n(\bar{s}_N)|^2 + |u(s)-u(\bar{s}_N)|^2\big) ds
 \Big)^{\frac{1}{2}}
\nonumber\\
&\leq C_1 \; 2^{-\frac{N}{4}} \, ,\\
I_{n,N}^3 &\leq C_0 \Big(\int_0^T |u(s)-u(\bar{s}_N)|^2 ds \Big)^{\frac{1}{2}}
 \Big( \int_0^T |h_n(s)-h(s)|_0^2\, ds\Big)^{\frac{1}{2}} \leq C_3\;  2^{-\frac{N}{4}}\, ,
\label{estim2}\\
I_{n,N}^4 &\leq C_0  \Big[ 1+ \Big( \sup_{0\leq t\leq T}  |u(t)|\Big)\,  \sup_{0\leq t\leq T}
\big( |u(t)|+|u_n(t)|\big)\Big]\,  2^{-\frac{N}{2}} \leq
C_4\;  2^{-\frac{N}{2}}\,  .
\label{estim3}
\end{align}
Furthermore, condition {\bf (C4)} implies that
\begin{equation} \label{eqHolderdet}
I_{n,N}^2 \leq C_0 2^{-N\gamma}\, \sup_{0\leq t\leq T}\big(|u(t)|+|u_n(t)|\big)
 \int_0^T (1+\|u(s)\|)  (|h(s)|_0+
|h_n(s)|_0)\, ds  \leq C_2\, 2^{-N \gamma}.
\end{equation}
For fixed $N$ and $k=1, \cdots, 2^N$,
 as $n\to \infty$, the weak convergence of $h_n$ to $h$ implies
that of $\int_{t_{k-1}}^{t_k} (h_n(s)-h(s))ds$ to 0 weakly in $H_0$.
 Since $\s(u(t_k))$ is a compact operator,
we deduce that for fixed $k$ the sequence $\s(u(t_k)) \int_{t_{k-1}}^{t_k} (h_n(s)-h(s))ds$
converges to 0 strongly in $H$ as $n\to \infty$.
Since $\sup_{n,k}   |U_n(t_k)|\leq 2 \sqrt{\bar{C}}$, we have
 $\lim_n I_{n,N}^5=0$. Thus
 \eqref{errorbound}--
\eqref{eqHolderdet} yield  for every integer $N\geq 1$
\[
\limsup_{n \to \infty}\left\{  \sup_{ t\leq T} |U_n(t)|^2 +\int_0^T \|U_n(t)\|^2\, dt
\right\}\le C 2^{-N (\gamma \wedge \frac{1}{4})}. 
\]
Since $N$ is arbitrary, we deduce that $\|U_n\|_X \to 0$ as $n\to \infty$.
This shows that every sequence in $K_M$ has a convergent
subsequence. Hence $K_M$ is  a sequentially relatively compact subset of $X$.
Finally, let $\{u_n\}$ be a sequence of elements of $K_M$ which converges to $v$ in $X$.
 The above argument
shows that there exists a subsequence $\{u_{n_k}, k\geq 1\}$
which converges to some element $u_h\in K_M$
for the same topology of $X$.
 Hence $v=u_h$,
$K_M$ is a closed subset of $X$, and this completes the proof of the proposition.
\end{proof}

\noindent {\bf Proof of Theorem \ref{PGDue}:}
 Propositions \ref{compact} and  \ref{weakconv} imply that the family
$\{u^\e\}$ satisfies the Laplace principle,  which is equivalent
to the large deviation principle, in $X$ with the  rate function defined by \eqref{ratefc};
 see Theorem 4.4 in \cite{BD00} or
Theorem 5 in \cite{BD07}. This  concludes the proof of Theorem \ref{PGDue}.
\hfill $\Box$

\section{Appendix : Proof of the Well posedness and apriori bounds}\label{App}

The aim of this section is to prove Theorem \ref{th3.1}. We at first suppose that conditions (C1)-(C3) are satisfied.
Let
$F :[0,T]\times  V\to V'$ be defined by
 \[  F(t,u)= -A u -B(u,u) -\tilde R( t,u)\; , \quad \forall t\in [0,T],\; \forall u\in V.\]
To lighten notations,  we suppress the dependence of $\s$, $\tilde{\s}$,  $\tilde R$  and $F$
on $t$.
The inequality \eqref{diffB1} implies that any $\eta >0$ there exists $C_\eta >0$ such that for $u,v\in V$,
\begin{equation} \label{diffF}
\langle F(u)-F(v)\, ,\, u-v\rangle \leq -(1-\eta) \|u-v\|^2 + \big( R_1 + C_\eta \|v\|_{\HH}^4\big) \, |u-v|^2.
\end{equation}
Let $\{\varphi_n \}_{n\geq1}$ be an orthonormal basis of the
Hilbert space $H$ such that    $\varphi_n\in Dom(A)$.
  For any $n\geq 1$, let  $ H_n = span(\varphi_1, \cdots, \varphi_n) \subset
Dom(A)$    and $P_n: H \to H_n$ denote the orthogonal projection from  $H$
onto $H_n$,  and finally  let  $\s_n=P_n\s$ and   $\tilde{\s}_n=P_n \tilde{\s}$.
Since $P_n$ is a contraction of $H$, from \eqref{LQ-norm} we deduce that
$|\s_n(u)|_{L_Q}^2 \leq |\s(u)|_{L_Q}^2$.
\par
 For   $h \in \mathcal{A}_M$,
consider the following stochastic ordinary differential equation
on the $n$-dimensional space $H_n$  defined  by  $u_{n,h}(0)=P_n \xi$,  and for
 $v \in H_n$:
\begin{eqnarray} \label{unh}
d(u_{n,h}(t), v)=\big[ \langle F(u_{n,h}(t)),v\rangle +(\tilde{\s}(u_{n,h}(t))h(t),
v) \big]dt +  (\s (u_{n,h}(t)) dW_n(t), v),
\end{eqnarray}
where $W_n(t)$ is defined in \eqref{W-n}.
 Then for $k=1, \, \cdots, \, n$ we have
\begin{eqnarray*}
d(u_{n,h}(t), \varphi_k)&=&\big[ \langle F(u_{n,h}(t)),\varphi_k \rangle
+(\tilde{\s}(u_{n,h}(t))h(t), \varphi_k) \big]dt\\
 && + \sum_{j=1}^n q_j^{\frac{1}{2}}  \big( \s (u_{n,h}(t))e_j\, ,\, \varphi_k \big)  d\beta_j(t).
\end{eqnarray*}
Note that  for $v\in H_n$   the map
$u   \in H_n  \mapsto \langle Au +\tilde R(u) \, ,\, v\rangle $ is
globally Lipschitz uniformly in $t$, while using   \eqref{boundB1}  we deduce that
  the map $  u \in H_n \mapsto \langle B(u)\, ,\, v\rangle  $ is
locally Lipschitz.
 Furthermore, since there exists some constant $C(n)$
 such that $\|v\| \leq C(n) |v|$ for $v\in H_n$,
 Conditions ({\bf C1}) and ({\bf C2}) imply that the
map  $u\in H_n  \mapsto \big( (\s_n(u) e_j\, ,\, \varphi_k) : 1\leq
j,k\leq n \big)$,  respectively
  $u\in H_n  \mapsto \big( (\tilde{\s}_n(u) h(t)\, ,\, \varphi_k) : 1\leq k\leq n \big) $,
is   globally Lipschitz from $H_n$ to $n\times n$ matrices,
respectively  to $\RR^n$ uniformly in $t$.
 Hence by a well-known result about existence and
uniqueness of solutions to
stochastic  differential equations  \cite{Kunita}, there
exists a maximal solution  $u_{n,h}=\sum_{k=1}^n (u_{n,h}\, ,\, \varphi_k\big)\, \varphi_k$ to \eqref{unh},
i.e., a stopping time
$\tau_{n,h}\leq T$ such that \eqref{unh} holds for $t<
\tau_{n,h}$ and as $t \uparrow \tau_{n,h}<T$,
$|u_{n,h}(t)| \to \infty$.
\par
The following proposition shows that  $\tau_{n,h}=T$   a.s. It  gives
estimates on $u_{n,h}$ depending only on
$T$, $M$, $K_i$, $L_i$ and  $\EX |\xi|^{2p}$,  which are valid for all $n$ and all
$K_2 \in  [0,  \bar{K}_2]$ for some $\bar{K}_2>0$.
Its proof depends on the following version of Gronwall's lemma (see \cite{DM}, Lemma 3.9  for the proof).

\begin{lemma} \label{lemGronwall}
Let $X$, $Y$,  $I$ and  $\varphi$ be  non-negative processes and
 $Z$ be a non-negative integrable random variable. Assume that
$I$ is non-decreasing and there exist non-negative constants
$C$, $\alpha, \beta, \gamma, \delta$ with the following  properties
\begin{equation} \label{Grw-cond}
\int_0^T \varphi(s)\, ds \leq C\; a.s.,\quad 2\beta e^C\le 1,\quad 2\delta e^C\le \alpha,
\end{equation}
  and  such that
 for $0\leq t\leq T$,
\begin{eqnarray*} 
X(t)+ \alpha Y(t) & \leq & Z + \int_0^t \varphi(r)\, X(r)\,  dr + I(t),\;
\mbox{\rm a.s.},\\
\EX(I(t)) &\leq & \beta\, \EX (X(t)) + \gamma  \int_0^t \EX (X(s))\, ds
 + \delta\, \EX (Y(t))  + \tilde{C},
\end{eqnarray*}
where  $\tilde{C}>0$ is a constant.
If $ X \in L^\infty([0,T] \times \Omega)$, then
we have
\begin{equation} \label{Gronwall}
\EX \big[ X(t) + \alpha Y(t)\big]  \leq 2 \,  \exp\big( {C+2 t  \gamma  e^C }\big)\,
 \big( \EX(Z) +\tilde{C} \big), \quad t\in [0,T].
\end{equation}
\end{lemma}

The following proposition provides the (global) existence and uniqueness
of approximate solutions and also their uniform (a priori) estimates.
This is the main preliminary step in the proof of Theorem~\ref{th3.1}.
As in Lemma \ref{timeincrement} we can
made this step under less restrictive  growth
conditions concerning $\tilde\s$ than \eqref{tilde-s-b} in ({\bf C3}).
\begin{prop} \label{Galerkin}  Fix $M>0$, $T>0$ and let Conditions ({\bf C1})--({\bf C3})
 be in force with the assumption \eqref{tilde-s-bv}
instead of \eqref{tilde-s-b}.
For any integer $p\geq 1$  there exists
$\bar{K}_{2} = \bar{K}_{2}(p,T,M)$
(which also depends on $K_i$, $\tilde{K}_i$ and $R_i$,  $i=0,1$, and on
$\tilde{K}_2$ if $\tilde{K}_2\neq K_2$)
such that  the
following result holds.
 Let $h\in \mathcal{A}_M$, $0\leq K_2\leq  \bar{K}_{2}$ and $\xi \in L^{2p}(\Om, H)$.
 Then equation \eqref{unh} has a unique
solution on $[0,T]$ (i.e.\ $\tau_{n,h}=T$ a.s.) with a modification  $u_{n,h} \in C([0, T], H_n)$ and
satisfying
\begin{align} \label{Galerkin1}
\sup_n \EX \,\Big(\, & \sup_{0\leq t\leq T}|u_{n,h}(t)|^{2p}  +
\int_0^T
\|u_{n,h}(s)\|^2 \, |u_{n,h}(s)|^{2(p-1)} ds \, \Big)
 \leq  C \big( \EX|\xi|^{2p} +1\big)
\end{align}
for some positive constant $C$ (depending on
 $p,K_i,\tilde{K}_i, i=0,1,2,  R_j, j=0,1, T,M$).
\end{prop}

\begin{proof}
Let $u_{n,h}(t)$ be the approximate maximal solution to \eqref{unh} described above.
 For every $N>0$, set
\begin{equation*}
\tau_N = \inf\{t:\; |u_{n,h}(t)| \geq N \}\wedge T.
\end{equation*}
Let $\Pi_n : H_0\to H_0$ denote the projection
operator defined by
$\Pi_n u =\sum_{k=1}^n \big( u \, ,\, e_k\big) \, e_k$, where
 $\{e_k, k\geq 1\}$ is the orthonormal basis of $H$  made by
 eigen-elements of the covariance operator
 $Q$ and used in  \eqref{W-n}.
\par
 It\^o's
 formula and the antisymmetry relation in \eqref{as}
yield that for $t \in [0, T]$,
\begin{align} \label{Galerkin3}
& |u_{n,h}(t\wedge \tau_N)|^2 = |P_n \xi|^2
+ 2 \int_0^{t\wedge \tau_N}\!\!
 \big( \sigma_n(u_{n,h}(s))  dW_n(s) , u_{n,h}(s)\big)
-2\int_0^{t\wedge \tau_N}\!\! \|u_{n,h}(s)\|^2 ds
 \nonumber \\
& \; -2\int_0^{t\wedge \tau_N}\!\! \big(\tilde{ R}  (u_{n,h}(s))-
 \tilde{\sigma}_n(u_{n,h}(s)) h(s), u_{n,h}(s)\big) \, ds
 +  \int_0^{t\wedge \tau_N} \!\! |\sigma_n(u_{n,h}(s))\, \Pi_n|_{L_Q}^2\, ds.
\end{align}
Apply again It\^o's formula for $f(z)=z^{p}$  when  $p\geq 2$ and
$z=|u_{n,h}(t\wedge \tau_N)|^2$. This yields for $t \in [0, T]$, and
any integer $p\geq 1$ (using the convention $p(p-1) x^{p-2}=0$ if
$p=1$)
\begin{equation} \label{estimate1}
|u_{n,h}(t\wedge \tau_N)|^{2p}  + 2p \int_0^{t\wedge \tau_N} \!\! |u_{n,h}(r)|^{2(p-1)} \,
   \|u_{n,h}(r)\|^2 \, dr
\leq \;   |P_n\xi|^{2p}  + \sum_{1\leq j\leq 5} {T}_j(t),
\end{equation}
where
\begin{eqnarray*}
{T}_1(t) &= & 2p\,    \int_0^{t\wedge \tau_N} \left(R_0+R_1 |u_{n,h}(r)|\right)
|u_{n,h}(r)|^{2p-1} dr,   \\
{T}_2(t) &= & 2p\;  \Big| \int_0^{t\wedge
\tau_N}
\big(\s_n(u_{n,h}(r))\; dW_n(r),
u_{n,h}(r)\big )\; |u_{n,h}(r) |^{2(p-1)} \Big| , \\
{T}_3(t) &= & 2p \,  \int_0^{t\wedge \tau_N}
|(\tilde{\s}_n(u_{n,h}(r))\; h(r), u_{n,h}(r))| \; |u_{n,h}(r)|^{2(p-1)}
dr, \\
T_4(t) &= &  p\,  \int_0^{t\wedge \tau_N}
|\s_n(u_{n,h}(r))\; \Pi_n |^2_{L_Q} \; |u_{n,h}(r) |^{2(p-1)}  dr,  \\
{T}_5(t) &= & 2p (p-1)  \,  \int_0^{t\wedge \tau_N}
|\Pi_n\s_n^*(u_{n,h}(r))\;  u_{n,h}(r)|^2_{0}\, |u_{n,h}(r)|^{2(p-2)} dr.
\end{eqnarray*}
Since $h\in \mathcal{A}_M$, the Cauchy-Schwarz inequality and
 condition \eqref{tilde-s-bv}
  imply that
\begin{align*}
 {T}_3 (t)& \leq \; 2p  \;   \int_0^{t\wedge \tau_N}\!\!
\Big( \sqrt{\tilde{K}_0} +
\sqrt{\tilde{K}_1} \, |u_{n,h}(r)| +  \sqrt{\tilde{K}_2} \,
\|u_{n,h}(r)\| \Big) \,|h(r)|_0 \,|u_{n,h}(r)|^{2p-1} dr  \nonumber \\
&  \leq \;  \frac{ p}{2}  \,  \int_0^{t\wedge \tau_N}
\!\!\|u_{n,h}(r)\|^2
 \, |u_{n,h}(r)|^{2(p-1)}\, dr
+  2p\, \tilde{K}_2  \, \int_0^{t\wedge \tau_N} \!\! |h(r)|_0^2 \,
|u_{n,h}(r)|^{2p}\, dr
\nonumber \\
&\quad  + 2p\, \sqrt{\tilde{K}_1} \, \int_0^{t\wedge \tau_N}\!\! |h(r)|_0 \, |u_{n,h}(r)|^{2p} dr
 + 2p\, \sqrt{\tilde{K}_0}
\, \int_0^{t\wedge \tau_N} \!\! |h(r)|_0|u_{n,h}(r)|^{2p-1} dr.
\end{align*}
Therefore using the inequality $|u|^{2p-1}\le 1+ |u|^{2p}$
to bound the last term
 we obtain
\begin{align} \label{estimate5}
& {T}_3 (t) \leq  \;  \frac{p}{2} \,  \int_0^{t\wedge \tau_N} \!\!\|u_{n,h}(r)\|^2
 \, |u_{n,h}(r)|^{2(p-1)}\, dr +
  2p\, \sqrt{\tilde{K}_0}  \, \int_0^{t} \!\!|h(r)|_0 dr \nonumber \\
&\qquad +  2p\,   \, \int_0^{t\wedge \tau_N} \!\!
\left[\left(\sqrt{\tilde{K}_0}+
\sqrt{\tilde{K}_1}\right) |h(r)|_0+
 \tilde{K}_2|h(r)|_0^2 \right]  |u_{n,h}(r)|^{2p} dr.
\end{align}
Using condition ({\bf C2}), relation \eqref{LQ-norm} and also the fact that
\begin{equation*}
 \| \s (u)\|_{\cL(H_0,H)} = \| \s^*
(u)\|_{\cL(H,H_0)}\le |\s (u)|_{L_Q},
\end{equation*}
  we deduce that
\begin{align}
 {T}_4(t) + {T}_5(t) & \leq  (2p^2 -p) \, K_2\,   \int_0^{t\wedge \tau_N}
 \|u_{n,h}(r)\|^2\, |u_{n,h}(r)|^{2(p-1)}\, dr \nonumber \\
& \quad + (2p^2-p) \,  \int_0^{t\wedge \tau_N} \!\! \Big( K_1 \, |u_{n,h}(r)|^{2p}
+  K_0  |u_{n,h}(r)|^{2(p-1)}\Big)\, dr \nonumber \\
& \leq  (2p^2 -p) \, K_2\,   \int_0^{t\wedge \tau_N}
 \|u_{n,h}(r)\|^2\, |u_{n,h}(r)|^{2(p-1)}\, dr
 \nonumber \\ & \quad
+ c_p \,  \int_0^{t\wedge \tau_N} \!\! \left[ K_0+ \left( K_0+K_1\right)
|u_{n,h}(r)|^{2p}\right]\, dr. \nonumber
\end{align}
Consequently \eqref{estimate1} for $K_2\le(4p-2)^{-1}$ yields
\begin{align} \label{estimate1-gal}
&|u_{n,h}(t\wedge \tau_N)|^{2p}  +
p \int_0^{t\wedge \tau_N} \!\! |u_{n,h}(r)|^{2(p-1)} \,
   \|u_{n,h}(r)\|^2 \, dr  \\
&\qquad \leq \;   |P_n\xi|^{2p}  + c_p \left[ (K_0+R_0)\, T+ \sqrt{\tilde K_0}
\int_0^T\!\!|h(r)|_0 dr \right]+
\int_0^{t\wedge \tau_N} \!\! \varphi (r)
|u_{n,h}(r)|^{2p}\, dr+I(t) \nonumber
\end{align}
for $t\in [0,T]$,
where $I(t)=\sup_{0\leq s\leq t}|T_2(s)|$ and
\[
\varphi (r)=c_p\left(R_0+ R_1 +K_0+ K_1 + \left[\sqrt{\tilde{K}_0}+
\sqrt{\tilde{K}_1}\right] |h(r)|_0+
 \tilde{K}_2|h(r)|_0^2 \right)
\]
for some constant $c_p>0$.
The Burkholder-Davies-Gundy inequality, ({\bf C2}) and Schwarz's inequality
yield   that for $t\in[0, T]$ and $\beta>0$,
\begin{align} \label{estimate3}
\EX I(t) =  &\EX \Big(\sup_{0\leq s\leq t}|T_2(s)|\Big) \; \leq \; 6p \,
\EX \Big\{ \int_0^{t\wedge \tau_N}
|u_{n,h}(r)|^{2(2p-1)} \; |\s_{n}(u_{n,h}(r))\; \Pi_n |^2_{L_Q}\;
 dr \Big\}^\frac12      \nonumber \\
\; \leq & \; \beta \;\EX  \Big(\sup_{0\leq s\leq t\wedge\tau_N}
|u_{n,h}(s)|^{2p}\Big)
+  \frac{9 p^2 K_2 }{\beta}   \;\EX
 \int_0^{t\wedge \tau_N} \| u_{n,h}(r)\|^2\,  |u_{n,h}(r)|^{2(p-1)} dr \nonumber \\
&\quad  + \frac{9 p^2 (K_0+K_1)}{\beta} \EX  \int_0^{t\wedge \tau_N}
|u_{n,h}(r)|^{2p} dr+   \frac{9 p^2 K_0}{\beta}\;T.
\end{align}
Thus we can apply Lemma~\ref{lemGronwall} for
\begin{equation}\label{XY}
X(t)= \sup_{0\leq s\leq t\wedge\tau_N} |u_{n,h}(s)|^{2p},\quad
Y(t)=  \int_0^{t\wedge \tau_N} \| u_{n,h}(r)\|^2\,  |u_{n,h}(r)|^{2(p-1)} dr.
\end{equation}
All inequalities for the parameters (see \eqref{Grw-cond}) can be achieved by choosing $K_2$ small enough.
Thus there exists $\bar{K}_2$ such
 that for $0\leq K_2 \leq \bar{K}_2$  we have
\[
\sup_n \EE\Big( \sup_{0\leq s\leq \tau_N} |u_{n,h}|^{2p} + \int_0^{\tau_N} \|u_{n,h}(s)\|\,
|u_{n,h}(s)|^{2(p-1)}\, ds \Big) \leq C(p)
\]
for all $n$ and $p$, where the constant $C(p)$ is independent of $n$.
\par
Now we are in position to conclude the proof of Proposition~\ref{Galerkin}.
As $N \to \infty$, $\tau_N \uparrow \tau_{n, h}$,  and on the set
$\{\tau_{n, h} < T  \}$, we have
 $  \sup_{0\leq s \leq   \tau_N}|u_{n,h}(s)| \to \infty$.
 Hence $\PX (\tau_{n, h} < T)=0$ and for almost all $\om$, for
 $N(\om)$ large enough,  $\tau_{N(\om)}(\om)=T$ and
$u_{n,h}(.)(\om) \in C([0, T], H_n)$.
By the Lebesgue monotone convergence theorem, we complete the
proof.
\end{proof}
\begin{remark}
{\rm If $\tilde K_2\le 2 -\delta$ for some $\delta>0$
then the bound $\bar{K}_2$ does not depend on $\tilde{K}_2$.
Indeed,
one may slightly change the proof by replacing the inequality
in \eqref{estimate5} by the following
\begin{align*}
 {T}_3 (t)& \leq
 \;  p\,  \tilde K_2 \,  \int_0^{t\wedge \tau_N} \!\!\|u_{n,h}(r)\|^2
 \, |u_{n,h}(r)|^{2(p-1)}\, dr
\\ &
 +  c_p\,   \, \int_0^{t\wedge \tau_N} \!\!\left[\sqrt{\tilde{K}_0} |h(r)|_0+
\left(\left[\sqrt{\tilde{K}_0}+
\sqrt{\tilde{K}_1}\right] |h(r)|_0+
 |h(r)|_0^2 \right)  |u_{n,h}(r)|^{2p}\right] dr.
\end{align*}
Therefore we can exclude the first term of right hand side  and obtain
relation \eqref{estimate1-gal} with coefficients independent of $\tilde{K}_2$.
}
\end{remark}

\par
\noindent \textbf{Proof of Theorem \ref{th3.1}: }\\
Let us at  first suppose  that condition (C3) is satisfied.
Let $\Om_T = [0, T]\times \Om$ be endowed with the product measure
$ds \otimes d\PX$ on $ \mathcal{B} ([0, T]) \otimes \mathcal{F}$.
Let $\bar{K}_2$ be defined by Proposition \ref{Galerkin} with $p=2$.
The inequalities \eqref{Galerkin1} and \eqref{interpol} imply that
for $K_2\in [0, \bar{K}_2]$ we have the following additional a priori estimate
\begin{equation} \label{boundHH}
  \sup_n \EX  \int_0^T \|u_{n,h}(s) \|^4_{\mathcal H} ds
\leq  C_{2}   (1+ \EX|\xi|^4).
\end{equation}
  The proof consists of several steps.
\medskip

\noindent \textbf{Step 1: } \;
The inequalities \eqref{Galerkin1}
and \eqref{boundHH} imply the existence of a subsequence of
$(u_{n,h})_{n\geq 0}$ (still denoted by the same notation),
 of  processes
 \[
 u_h \in {\mathcal X}:= L^2(\Om_T, V) \cap L^4(\Om_T,
{\mathcal H}) \cap L^4(\Omega , L^\infty([0, T], H)),
\]
  $F_h \in L^2(\Om_T, V')$ and
 $ S_h, \tilde{S}_h \in L^2(\Om_T, L_Q)$,  and finally  of random variables
 $\tilde{u}_h(T) \in
L^2(\Om, H)$,
for which the following  properties hold:
\\
(i) $u_{n,h}  \to u_h$ weakly in $L^2(\Om_T, V)$, \\
(ii) $u_{n,h}  \to u_h$ weakly in $L^4(\Om_T, {\mathcal H})$, \\
(iii)    $u_{n,h} $ is weak star converging to $ u_h$  in $L^4(\Omega, L^\infty([0, T],H))$, \\
(iv) $u_{n,h}(T)  \to
 \tilde{u}_h(T)
$ weakly in $L^2(\Om, H)$, \\
(v) $F(u_{n,h}) \to F_h$ weakly in $L^2(\Om_T, V')$,   \\
(vi)   $\s_n(u_{n,h}) \Pi_n \to S_h$   weakly in $L^2(\Om_T,L_Q)$,   \\
(vii) $\tilde{\s}_n(u_{n,h}) h \to \tilde{S}_h$
weakly in $L^{\frac{4}{3}}(\Omega_T,H)$.
\par
Indeed, (i)-(iv) are straightforward consequences of
 Proposition \ref{Galerkin}, of  \eqref{boundHH}, and of
 uniqueness of the limit
of $\EX \int_0^T (u_{n,h}(t), v(t))dt$ for appropriate $v$.
Furthermore, given $v \in L^2(\Om_T, V)$, we have $Av\in L^2(\Om_T , V')$.
Since for $u,v\in L^2(\Om_T,V)$, $\EX \int_0^T  \langle Au(t)\,, \, v(t)\rangle\, dt =
\EX \int_0^T  \langle u(t)\, , \, Av(t)\rangle\, dt$,
\begin{equation} \label{eq3.14}
 \EX \int_0^T  \langle A u_{n,h}(t), v(t)\rangle \, dt
\to \;
\EX \int_0^T \langle A u_h(t)\, ,\, v(t)\rangle \,  dt .
\end{equation}
Using \eqref{Galerkin1} with $p=2$,   \eqref{boundB}, \eqref{boundHH}, condition {\bf (C3)},
the Poincar\'e and
  the Cauchy-Schwarz inequalities,  we  deduce
\begin{align*}
 \sup_n &\;   \EX  \int_0^T \big|
 \langle  B(u_{n,h}(t)), v(t) \rangle
 + (\tilde{R} (u_{n,h}(t)), v(t))\big|\,  dt   \\
&\leq \; C_1\,  \sup_n  \left\{ \EX \int_0^T
  \|u_{n,h}(t)\|^4_\HH  +   \EX \int_0^T|u_{n,h}(t)|^2 dt\right\}
+  \; C_2 \EX \int_0^T(1+ \|v(t)\|^2) dt
 \\
 &\leq \;  C_3 \left(1+ E|\xi|^4 +
 \EX \int_0^T  \|v(t)\|^2  dt\right)<+\infty .
\end{align*}
Hence $\{B(u_{n,h}(t)) + \tilde{R} (u_{n,h}(t))\,, \,  n\geq 1 \} $ has a subsequence
converging weakly in $L^2(\Om_T, V')$, which completes the proof of  (v).
\par
Since $\Pi_n$ contracts the  $|\cdot |$ norm,
({\bf C2}), \eqref{LQ-norm}
and \eqref{Galerkin1} for $p=2$ prove that
(vi) is a straightforward of the following
\[
\sup_n \EX\! \int_0^T\!\! |\s_n(u_{n,h}(t)) \Pi_n|^2_{L_Q} dt \leq K_0T +
\sup_n \EX  \!  \int_0^T \!\!  \left(K_1 |u_{n,h}(t)|^2 + K_2 \|u_{n,h}(t)\|^2\right) dt
<  \infty.
\]
Finally, using \eqref{tilde-s-b} in {\bf (C3)}, H\"older's inequality, (\ref{Galerkin1})
 with $p=2$ and \eqref{boundHH},
we deduce
\begin{align*}
\EX\int_0^T |\tilde{\s}_n( & u_{n,h}  (s)\, h(s)  |^{\frac{4}{3}} \, ds \leq
\EX\int_0^T \big[\sqrt{ \tilde{K}_0}  + \sqrt{ \tilde{K}_1} |u_{n,h}(s)|
+ \sqrt{ \tilde{K}_\HH} \|u_{n,h}\|_{\mathcal H}\big]^{\frac{4}{3}} \; |h(s)|_0^{\frac{4}{3}} \, ds \\
&\leq C_1 \Big(\EX\int_0^T |h(s)|_0^2\, ds \Big)^{\frac{2}{3}} \Big(\EX\int_0^T \big[1
+ |u_{n,h}(s)|^4 + \|u_{n,h}(s)\|_\HH^4\big]  ds \Big)^{\frac{1}{3}}\leq C(M,T)
\end{align*}
for every integer $n\geq 1$. This completes the proof of (vii).

\medskip

\noindent \textbf{Step 2: }
For $\delta >0$, let $f \in H^1(-\delta, T+\delta)$ be such that
 $\|f\|_\infty = 1$,
 $f(0)=1$ and for any integer $j \geq 1$ set
 $g_j(t)=f(t)\varphi_j$,
 where $\{\varphi_j \}_{j\geq 1}$
 is
the previously chosen orthonormal basis for $H$. The It\^o formula
implies that for any $j\geq 1$, and for $0 \leq t \leq T$,
\begin{eqnarray} \label{equality1}
  \big( u_{n,h}(T)\, ,\, g_j(T)\big) = \big( u_{n,h}(0)\, ,\,  g_j(0)\big)
 +\sum_{i=1}^4  I_{n ,j}^i,
\end{eqnarray}
where
\begin{eqnarray*}
I_{n ,j}^1  =  \int_0^T (u_{n,h}(s), \varphi_j) f'(s) ds,&&
I_{n ,j}^2  =   \int_0^T \big( \s_n(u_{n,h}(s)) \Pi_n dW(s), g_j(s)\big),  \\
I_{n ,j}^3  =  \int_0^T \langle F(u_{n,h}(s)), g_j(s)\rangle  ds , &&
I_{n ,j}^4  =  \int_0^T \big( \tilde{\s}_n(u_{n,h}(s)) h(s), g_j(s)\big) ds.
\end{eqnarray*}
Since $f' \in L^2([0, T])$ and  for every $X \in L^2(\Om)$,
  $(t, \om) \mapsto \varphi_j  X(\omega)\, f'(t)
 \in L^2(\Om_T, H)$,   (i)~above implies that
as $n \to \infty$, $I_{n ,j}^1 \to \int_0^T (u_{h}(s), \varphi_j)
f'(s) ds$ weakly in $L^2(\Om)$. Similarly, (v) implies that as
$n\to \infty$, $I_{n ,j}^3 \to \int_0^T \langle F_h(s), g_j(s)\rangle ds$
weakly in $L^2(\Om)$, while (vii) implies that $I_{n ,j}^4 \to \int_0^T \big(
\tilde{S}_h(s), g_j(s)\big) ds$ weakly in $L^{\frac{4}{3}}(\Omega)$.
\par
To prove the convergence of $I_{n ,j}^2$, as in \cite{Sundar} (see also \cite{DM}), let
$ \mathcal{P}_T$ denote the class of predictable processes in
$L^2(\Om_T, L_Q(H_0, H))$ with the inner product
\begin{eqnarray*}
 (G, J)_{\mathcal{P}_T} =\EX \int_0^T \big(G(s), J(s)\big)_{L_Q} ds=
 \EX \int_0^T trace (G(s)QJ(s)^*) ds.
\end{eqnarray*}
The map  $\mathcal{T}: \mathcal{P}_T \to  L^2(\Omega)$
 defined by $ \mathcal{T} (G)(t) =
 \int_0^T \big( G(s)  dW(s) , g_j(s)\big) $
 is linear and continuous  because of the It\^o isometry.
 Furthermore, (vi) shows that for every  $G \in \mathcal{P}_T$, as $n\to \infty$,
  $\big(\s_n(u_{n,h}) \Pi_n, G\big)_{\mathcal{P}_T} \to (S_h,  G)_{\mathcal{P}_T}$
weakly in $L^2(\Omega)$.
\par
Finally, as $n\to \infty$, $P_n\xi =u_{n,h}^\e(0) \to
\xi $ in $H$ and by (iv), $(u_{n,h}(T), g_j(T))$ converges to
$(\tilde{u}_h(T), g_j(T))$ weakly in $L^2(\Om)$. Therefore,
as $n\to \infty$, \eqref{equality1} leads to
\begin{align} \label{equality2}
  (\tilde{u}_{h}(T), \varphi_j)\, f(T)
 &= \; \big( \xi,\varphi_j\big)
 + \int_0^T \big(u_{h}(s), \varphi_j\big)
f'(s) ds   +  \int_0^T \big( S_h(s)dW(s), g_j(s) \big)  \nonumber \\
 &\qquad  + \int_0^T \langle F_h(s) , g_j(s)\rangle ds +\int_0^T \big(
 \tilde{S}_h(s) , g_j(s)\big)
ds.
\end{align}

For $\delta >0$, $k>\frac{1}{\delta}$, $t \in [0, T]$, let $f_k \in H^1(-\delta, T+\delta)$ be such
that
 $\|f_k\|_\infty =1$,
  $f_k=1$ on $(-\delta, t-\frac{1}{k})$ and $f_k=0$ on
$\big(t, T+\delta\big)$.
 Then $f_k \to 1_{(-\delta, t)}$ in
$L^2$, and $f'_k \to -\delta_t$ in the sense of distributions.
Hence as $k\to \infty$, \eqref{equality2} written with $f:=f_k$
yields
 \[
0= \big(\xi - {u}_{h}(t) ,\varphi_j\big)
  +   \int_0^t \big(S_h(s)dW(s), \varphi_j\big)
 +  \int_0^t  \langle F_h(s), \varphi_j\rangle ds
+\int_0^t \big(\tilde{S}_h(s) , \varphi_j\big) ds
\]
for almost all $t\in [0,T]$. This relation makes it possible to suppose
(after some modification) that ${u}_{h}(t)$ is weakly continuous in $H$
for almost all $\om\in\Om$. Now
note that $j$ is arbitrary and $\EX \int_0^T |S_h(s)|^2_{L_Q} ds <
\infty$;  we deduce that for $0\leq t \leq T$,
\begin{eqnarray} \label{equality3}
 u_{h} (t)  =   \xi
 +  \int_0^t  S_h(s)dW(s)  + \int_0^t F_h(s)  ds +\int_0^t  \tilde{S}_h(s)ds \in H.
\end{eqnarray}
Moreover  $\int_0^t F_h(s)  ds\in H $.
 Let $f=1_{(-\delta,
T+\delta)}$; using again  \eqref{equality2}  we obtain
\begin{eqnarray*}
 \tilde{u}_{h} (T)
  =   \xi
 +  \int_0^T  S_h(s)dW(s)  + \int_0^T F_h(s)  ds +\int_0^T \tilde{S}_h(s)ds.
\end{eqnarray*}
This equation and \eqref{equality3} yield that
 $\tilde{u}_h(T) = u_h(T)$ a.s.
\medskip
\par
\noindent \textbf{Step 3: }
In \eqref{equality3} we still have to prove that  $ds \otimes d\PX$ a.s. on $\Om_T$,  one has
\begin{eqnarray*}
 S_h(s)=\s(u_h(s)), \;     F_h(s)=F(u_h(s))\;\mbox{\rm and }\;
 \tilde{S}_h(s)=\tilde{\s}(u_h(s))\;h(s) .
\end{eqnarray*}
To establish these relations we use the same idea  as in \cite{Sundar}.
Let
\begin{eqnarray*}
v\in   {\mathcal X}=   L^4(\Om_T, \HH)\cap L^4\big( \Omega,
 L^\infty([0,T],H) \big) \cap L^2(\Omega_T,V)\;.
\end{eqnarray*}
Suppose that $L_2<2$ and let $0<\eta<\frac{2-L_2}{3}$;  for every $t \in [0, T]$,
set
\begin{eqnarray} \label{r}
 r (t)  = \int_0^t  \Big[\,2\,  R_1 +
2\, C_\eta \, \|v(s)\|^4_{\HH} + L_1  + 2\sqrt{\tilde{L}_1} |h(s)|_0
+  \frac{\tilde{L}_2}{\eta}  |h(s)|_0^2   \Big]\, ds,
\end{eqnarray}
where $C_\eta$  is a  function of $\eta$ such that \eqref{diffF} holds.
Then almost surely,  $0\le r(t) < \infty$ for all $t \in [0, T]$.
Moreover, we also have that
\begin{equation}\label{r-facts}
r\in L^1(\Omega,L^\infty(0;T)), \; e^{-r}\in L^\infty(\Omega_T),\;
 r'\in L^1(\Omega_T),\;
r'e^{-r}\in L^\infty(\Om, L^1((0,T)).
\end{equation}
Weak  convergence in (iv) and the property $P_n\xi\to\xi$ in $H$ imply that
\begin{equation}\label{l-lim}
\EX\big( |u_h(T)|^2\, e^{-r(T)}\big)-\EX |\xi|^2
\leq \liminf_n \left[\EX\big(|u_{n,h}(T)|^2\, e^{-r(T)}\big)-\EX |P_n\xi|^2\right].
\end{equation}
We now apply  It\^o's formula  to
$ |u(t)|^2\, e^{-r(t)}$ for $u=u_h$ and $u=u_{n,h}$.
This gives the relation
\[
\EX\big( |u(T)|^2\, e^{-r(T)}\big)-\EX |u(0)|^2=\EX\int_0^Te^{-r(s)}d\left\{|u(s)|^2\right\}
-\EX\int_0^T r'(s)e^{-r(s)}|u(s)|^2 ds,
\]
which can be justified due to \eqref{r-facts} and the property
$|u|^2\in L^1(\Om, L^\infty((0,T))$.
Using \eqref{equality3}, \eqref{unh}   and  letting $u = v  + (u-v)$
after simplification, from \eqref{l-lim} we obtain
\begin{align} \label{r1}
  \EX & \int_0^T \!\!  e^{-r(s)}\,  \big[ -r'(s)\big\{ \big|u_h(s)-v(s)\big|^2
+   2\big( u_h(s)-v(s)\, ,\, v(s)\big)\}
 + 2\langle F_h(s),u_h(s)\rangle \nonumber \\
&+ |S_h(s)|^2_{L_Q} + 2\big(\tilde{S}_h(s)\, ,\, u_h(s)\big)\big]\, ds
\leq \liminf_n  X_n,
\end{align}
where
\begin{align*}
& X_n=\EX\int_0^T e^{-r(s)}\big[  -r'(s) \big\{ \big|u_{n,h}(s)-v(s)\big|^2
+ 2\big( u_{n,h}(s)- v(s)\, ,\, v(s)\big) \big\}  \\
 & + 2\langle F(u_{n,h}(s)),u_{n,h}(s)\rangle
+ | \sigma_n (u_{n,h}(s))\Pi_n|^2_{L_Q} +
2\big(\tilde{\sigma}(u_{n,h}(s)) h(s)\, ,\, u_{n,h}(s)\big)\big]\,
ds .
\end{align*}
 The inequalities  in \eqref{diffF}, ({\bf C2}), ({\bf C3}),
and also   \eqref{r} and Schwarz's inequality
imply that
\begin{align} \label{r2}
Y_n &\; : = \; \EX \int_0^T e^{-r(s)}\big [-r'(s)
|u_{n,h}(s)- v(s)|^2    \nonumber \\ &
 +   2 \langle F(u_{n,h}(s))-F(v(s)), u_{n,h}(s)- v(s)\rangle
\;  +|\s_{n}(u_{n,h}(s))\; \Pi_n - \s_n(v(s)) \; \Pi_n|^2_{L_Q}
  \nonumber \\
&\;  + 2 \big(\big\{\tilde{\s}_{n}(u_{n,h}(s))- \tilde{\s}_n(v(s))\big\}\; h(s),
u_{n,h}(s)-v (s)\big) \Big] ds  \leq 0.
\end{align}
Furthermore, $X_n=Y_n + \sum_{i=1}^2 Z_n^i$, with
\begin{align*}
& Z_n^1 \, = \,   \EX  \int_0^T \!\!  e^{-r(s)} \Big[ -2  r'(s) (u_{n,h}(s))- v(s), v(s))
 + 2\langle F(u_{n,h}(s)), v(s) \rangle \\
\; &+ 2 \langle F(v(s)),u_{n,h}(s)\rangle
 -  2\langle F(v(s)), v(s)\rangle
+  2\big( \sigma_n(u_{n,h}(s)) \Pi_n \, , \, \sigma(v (s)\big)_{L_Q}
 \\ &
 \;  +  2\big( \tilde{\s}_{n}\big( u_{n,h}(s)\big)h(s), v(s)\big)
 +  2 \big( \tilde{\s} (v(s))\; h(s),
u_{n,h}(s)) -2 (P_n\tilde{\s}(v(s)) h(s), v(s) \big) \, \Big]\,  ds , \\
& Z_n^2  \, =  \EX  \int_0^T \!\! e^{-r(s)} \Big[
2\Big(\s_n(u_{n,h}(s)) \Pi_n, \big( [\s(v(s))
\Pi_n-\s(v(s))\big)\Big)_{L_Q}  -  |P_n \s(v(s)) \Pi_n|^2_{L_Q}\Big]
ds.
\end{align*}
The weak convergence properties (i)-(vii) imply that, as $n\to
\infty$,  $Z_n^1 \to Z^1$ where
\begin{align} \label{r3}
 Z^1 \,
 =   \, \EX & \int_0^T \!  e^{-r(s)} \big[ -2  r'(s) \big(u_{h}(s) - v(s), v(s)\big)
+  2\langle F_h(s),v(s)\rangle
+ 2 \langle F(v(s)), u_{h}(s)\rangle    \nonumber  \\
& - 2\langle F(v(s)), v(s)\rangle
 +2 \big( S_h(s)\, ,\, \sigma(v(s))\big)_{L_Q} +  2 (\tilde{S}_h(s), v(s))    \nonumber \\
& + 2 \big(\tilde{\s} (v(s))\; h(s), u_{h}(s)\big)
 -2 \big(\tilde{\s}(v(s)) h(s), v(s) \big)\, \big] ds.
\end{align}
As for  $Z_n^2$ we note that the Lebesgue dominated convergence theorem implies
 that
\[
\EX \int_0^T e^{-r(s)}   |\s(v(s)) (\Pi_n-Id_{H_0})|_{L_Q}^2ds \to 0~~\mbox{ as }~~
n\to \infty.
\]
Using once more the dominated Lebesgue convergence  theorem, we deduce that
\begin{equation} \label{Z2}
 Z_n^2 \to -  \EX \int_0^T e^{-r(s)} |\sigma(v(s))|_{L_Q}^2 ds~~\mbox{ as }~~
n\to \infty.
\end{equation}
Thus, \eqref{r1}-\eqref{Z2} imply that for  any $v \in{\mathcal X}$,
\begin{align} \label{r7}
&    \EX  \int_0^T e^{-r(s)}\Big\{ - r'(s) |u_h (s)- v(s)|^2
 + 2\langle F_h(s)-F(v(s)), u_h(s)-v(s)\rangle \nonumber \\
&\;   +  |S_h(s)-\s(v(s))|^2_{L_Q}
+ 2\Big( \tilde{S}_h(s)-\tilde{\s}(v(s)) h(s)\, ,\,  u_h(s)-v(s)\Big) \Big\} ds
\leq  0.
\end{align}
Let  $v=u_h \in {\mathcal X} $;
we conclude that $S_h(s) =\s(u_h (s)),$ $ ds\otimes d \PX$ a.e.
For  $\lambda\in \RR$, $\tilde{v} \in L^\infty(\Omega_T,V)$,
 set  $v_\lambda =u_h -\lambda
\tilde{v}$\, ; then it is clear that $v_\lambda\in {\mathcal X}$.
Applying  \eqref{r7}  to $v:=v_\lambda$ and
 neglecting $| \s(u_h(s)) - \s(v_\lambda(s))|^2_{L_Q}$, yields
\begin{align} \label{r8}
 \EX  \int_0^T e^{-r_\la(s)} & \Big[ -\lambda^2  r'_\la(s) |\tilde{v}(s)|^2
 + 2 \lambda  \Big\{\langle F_h(s)-F(v_\lambda(s)),
\tilde{v}(s)\rangle  \nonumber \\
&\quad
+ \Big(\tilde{S}_h(s)-\tilde{\s}(v_\lambda(s)) h(s), \tilde{v}(s)\Big) \Big\}
 \Big]\, ds \leq  0,
\end{align}
where $r_\la(s)$ is given by \eqref{r} with $v_\la$ instead of $v$.
Using ({\bf C3})  we obtain
\begin{align*}
\EX  & \int_0^T e^{-r_\la(s)}\big|\big( \big[\tilde{\s}(v_\lambda(s))- \tilde{\s}(u_h(s))]\,  h(s)\, ,\,
  \tilde{v}(s)\big)\big| ds \\
  & \leq
|\lambda | \, \EX  \int_0^T |h(s)|_{0} \,  |\tilde{v}(s)| \, \left(
\sqrt{\tilde{L}_1}  |\tilde{v}(s)| + \sqrt{\tilde{L}_2}\,
 \|\tilde{v}(s)\|   \right) \, ds  \to 0
\end{align*}
  as $\lambda \to 0$.
 Hence, by   the dominated convergence theorem,
\begin{eqnarray*}
\lim_{\lambda \to 0} \EX \int_0^T \!   e^{-r_\la(s)}\Big(
\tilde{S}_h(s)-\tilde{\s}(v_\lambda(s))
 h(s)\, , \,  \tilde{v}(s)\Big)  ds {}\qquad {}
 \\= \EX \int_0^T \!  e^{-r_0(s)}\Big( \tilde{S}_h(s)-\tilde{\s}(u_h(s))h(s)\, ,\,
 \tilde{v}(s)\Big)
 ds.
\end{eqnarray*}
Furthermore,     \eqref{diffF}   yields  for $\lambda \neq 0$ and $s\in [0,T]$
\begin{align*}
  \big|\langle  F(v_\lambda & (s)) -F(u_h(s)), \tilde{v}(s)\rangle\big| \leq
C \, |\lambda| \left[ |\tilde{v}(s)|^2 + \|\tilde{v}(s)\|^2
+ |\tilde{v}(s)|^2\,
  \| u_h(s)\|_\HH^4
\right].
\end{align*}
Thus  we deduce  as $\lambda \to 0$,
\begin{eqnarray*}
 \EX \int_0^T  e^{-r_\la(s)}\langle F_h(s)- F(v_\lambda(s)), \tilde{v}(s)\rangle
ds  \to \EX \int_0^T e^{-r_0(s)} \langle F_h(s)- F(u_h(s)),
\tilde{v}(s)\rangle ds.
\end{eqnarray*}
Thus, dividing \eqref{r8} by $\lambda>0$ (resp. $\lambda <0 $)
and letting $\lambda \to
0$ we obtain that for every
  $\tilde{v} \in   L^\infty(\Omega_T,V)$,
 which is a dense subset of $L^2(\Om_T, V)$,
\begin{eqnarray}
 \EX \int_0^T  e^{-r_0(s)}\Big[  \big\langle
   F_h(s)- F(u_h(s))\, , \, \tilde{v}(s)\big\rangle
  + \big( \tilde{S}_h(s)-\tilde{\s}(u_h(s)) h(s)\, ,\,
\tilde{v}(s)\big) \Big] \,  ds =
 0. \nonumber
\end{eqnarray}
Hence a.e. for $t \in [0, T]$, \eqref{equality3} can be rewritten
as
\begin{eqnarray} \label{r9bis}
u_{h} (t)  =   \xi
 +  \int_0^t  \s(u_h(s))dW(s)  +
 \int_0^t \big[F(u_h(s)) + \tilde{\s}(u_h(s)) h(s) \big]ds.
\end{eqnarray}
 Furthermore, (i)-(iii)  
 imply that
\begin{eqnarray} 
\EX\Big(  \int_0^T \|u_h(t)\|^2\, dt \Big)& \leq& \sup_n \EX
\int_0^T \|u_{n,h}(t)\|^2 dt \leq C
 \big( 1+E|\xi|^4\big),\nonumber \\
\EX\Big( \sup_{0\leq t\leq T} |u_h(t)|^4 \big) &\leq & \sup_n \EX
\Big(\sup_{0\leq t\leq T} |u_{n,h}(t)|^4 \Big)  \leq    C\, \big(
1+E|\xi|^4\big),  \nonumber \\
\label{boundgeneral3} \EX\Big(  \int_0^T \|u_h(t)\|^4_\HH\, dt
\Big)& \leq& \sup_n \EX \int_0^T \|u_{n,h}(t)\|_{\mathcal H}^4 \, dt \leq
C  \big( 1+E|\xi|^4\big).
\end{eqnarray}
This completes the proof of
 \eqref{eq3.1}.
\medskip\par
\noindent \textbf{Step 4: }
Now we prove that $u_{h}\in C([0,T],H)$ almost surely.
We first note that \eqref{r9bis} yields that
 $e^{-\de A}u_{h}\in C([0,T],H)$ a.s. for any $\de>0$.
Indeed, since for $\delta >0$ the operator $e^{-\delta A}$ maps $H$ to $V$ and $V^\prime$ to $H$,
we deduce that  $e^{-\delta A}\int_0^. F(u_h(s))\, ds $ belongs to $C([0,T],H)$.
Condition ({\bf C3}) implies that
$e^{-\delta A}\int_0^. \tilde{\s}(u_h(s))\, h(s)\, ds$ also belongs to
$C([0,T],H)$.
Finally, condition ({\bf C2})  implies
$\EE \int_0^T |e^{-\delta A} \sigma(u_h)(s)|^2_{L_Q}\, ds <+\infty$. Thus
$\int_0^. e^{-\delta A}  \s(u_h(s))\, dW(s)$ belongs to $C([0,T],H)$ a.s. (see e.g. \cite{PZ92}, Theorem 4.12).
 Therefore it  is sufficient to prove that
\begin{equation}\label{conv-cont}
\lim_{\delta\to0} \EX\left\{ \sup_{0\leq t\leq T} |u_{h}(t)-e^{-\de A}u_{h}(t)|^2\right\}=0.
\end{equation}
Let $G_\de= Id-e^{-\de A}$ and apply It\^o's   formula to
$|G_\de u_{h}(t)|^2$. This yields
\begin{eqnarray} \label{ito-g-de}
 |G_\de u_{h}(t)|^2 & = &  |G_\de \xi|^2  -2\int_0^{t}\!\! \| G_\de u_{h}(s)\|^2 ds
+ 2  I(t)   +    \int_0^{t} \!\! | G_\de \sigma(u_{h}(s))|_{L_Q}^2\, ds
 \nonumber \\
&& \; +2\int_0^{t}\!\! \big\langle  B(u_h(s))+  \tilde R (u_{h}(s)) +
  \tilde{\sigma}(u_{h}(s)) h(s),  G^2_\de u_{h}(s)\big\rangle \, ds,
\end{eqnarray}
where $I(t)=  \int_0^{t} \big( G_\de\sigma(u_{h}(s))  dW(s),  G_\de u_{h}(s)\big)$.
By the Burkholder-Davies-Gundy and Schwarz inequalities we have
\begin{eqnarray*}
\EX\sup_{0\leq t\leq T} |I(t)| &\le&
C\EX\left( \int_0^{T}\!\! |G_\de u_{h}(s)|^2
 |G_\delta \, \sigma(u_{h}(s))|^2_{L_Q} ds\right)^{1/2}  \\
&\leq  &  \frac{1}{2}\;  \EX\sup_{0\leq t\leq T}  |G_\de u_{h}(t)|^2 + \frac{C^2}{2} \EX
\int_0^{T} \!\! | G_\de \sigma(u_{h}(s))|_{L_Q}^2\, ds .
 \end{eqnarray*}
Hence  for some constant $C$, \eqref{ito-g-de} yields
\begin{align*}
\EX\sup_{0\le t\le T } | & G_\de u_{h}(t)|^2 \le  2 \, |G_\de \xi|^2
+ C\, \EX \int_0^{T}\!\! | G_\de \sigma(u_{h}(s))|_{L_Q}^2 ds \\
& \; +\, 4\,  \EX \int_0^{T}\!\!  \left|\big\langle  B(u_h(s))+  \tilde R (u_{h}(s)) +
  \tilde{\sigma}(u_{h}(s)) h(s),  G^2_\de u_{h}(s)\big\rangle \right|  \, ds .
\end{align*}
Since for every $u\in H$, $|G_\de u| \to 0$   as $\delta\to 0$
and $\sup_{\delta>0} |G_\delta|_{L(H,H)}\leq 1$,   we deduce that if $\{\varphi_k\}$
denotes an orthonormal basis in $H$, then
$ | G_\de \sigma(u_{h}(s))Q^{1/2} \varphi_k|^2 \to 0$
for every $k$  and almost every $(\omega,t)\in \Omega\times [0,T]$.
Since
\[
{\displaystyle \sup_{\delta>0 }
|G_\delta \sigma(u_h)|_{L_Q}^2 \leq
  \sum_k  \sup_{\delta >0}  | G_\de \sigma(u_{h})Q^{1/2} \varphi_k|^2
\leq C |\sigma(u_h)|_{L_Q}^2\in L^1(\Omega\times [0,T])},
\]
the Lebesgue  dominated convergence theorem
 implies that $  \EX \int_0^{T}\!\! | G_\de \sigma(u_{h}(s))|_{L_Q}^2 ds \to 0$.
Furthermore, given  $u\in V$ we have  $\|G_\de^2 u\| \to 0$ as $\delta\to 0$
and $\sup_{\delta >0} |G_\delta|_{L(V,V)}\leq 2 $.  Hence
 $ \big\langle  B(u_h(s))+  \tilde R (u_{h}(s)) +
  \tilde{\sigma}(u_{h}(s)) h(s),  G^2_\de u_{h}(s)\big\rangle \to 0$ for almost every
  $(\omega,s)$.
Therefore, as above, the Lebesgue
  dominated convergence theorem concludes the proof of \eqref{conv-cont}.

\medskip\par
\noindent \textbf{Step 5: }
To complete the proof of Theorem \ref{th3.1}, we  show
 that
the solution  $u_h$ to \eqref{r9bis} is unique in $X:=C([0, T], H) \cap L^2([0,
T], V)$.
 Let $v \in X$ be another solution to \eqref{r9bis} and
 \[
\tau_N =\inf \{t \geq 0: |u_h(t) | \geq N \} \wedge
\inf \{t \geq 0: |v(t) | \geq N \} \wedge T.
\]
Since $|u_h(.) |$ and $|v(.)|$ are
a.s. bounded on $[0,T]$,  we have    $\tau_N \to T$ a.s. as $N\to \infty$.
\par
Let $U = u_h-v$. By  It\^o's formula we have
\begin{align}\label{ito-psi}
 e^{-a \int_0^{t\wedge \tau_N}   \|u_h(r)\|^4_\HH dr}
 |U(t\wedge \tau_N)|^2 \;\; =  \int_0^{t\wedge \tau_N}\Psi(s) ds +\Phi(t\wedge \tau_N),
\end{align}
where
\begin{align*}
  \Psi(s) &
 =
  e^{- a\int_0^{s}
 \|u_h(r)\|^4_\HH dr}  \Big[ -a \|u_h(s)\|^4_\HH
 |U(s)|^2
 \\
&\;-2 \|U(s)\|^2
  - 2 \langle B(u_h(s))- B(v(s)), U(s)\rangle
 + |\s(u_{h}(s))-\s(v(s))|^2_{L_Q} \\
&\;  +  2 \big( [\tilde{\s}(u_h(s))-\tilde{\s}(v(s))] h(s), U(s)  \big)
- 2 \big( \tilde R(u_h(s))-\tilde R(v(s))\, ,
U(s)\big)   \Big]
\end{align*}
and
\begin{align*}
  \Phi(\tau)
 =2\int_0^\tau
  e^{- a\int_0^{s}
 \|u_h(r)\|^4_\HH dr}  \big( U(s),  \left[\s(u_{h}(s))-\s(v(s))\right] dW(s)\big).
\end{align*}
Now we
 set $a=2 {C}_\eta$ where ${C}_\eta$ is defined by \eqref{diffB1}.
Then using \eqref{diffB1} and Conditions ({\bf C2}) and ({\bf C3})
we obtain  that for some non negative constant $C(\eta)$  which depends on $\eta$,
$R_1$, $L_1$, $\tilde{L}_i$, $i=1,2$,
and is independent of $L_2$,
\begin{align}\label{Psi}
 \Psi(s)
\leq &\;    e^{- a\int_0^{s}
\|u_h(r)\|^4_\HH dr} \Big[   -(2-3\eta-L_2) \|U(s)\|^2 \nonumber \\
&  +
\Big(2 R_1+  L_1  + \frac{\tilde{L}_2}{\eta}
  |h(s)|_{0}^2  + 2\,\sqrt{\tilde{L}_1}|h(s)|_0\Big) |U(s)|^2  \Big] \nonumber \\
\leq &\;    e^{- a\int_0^{s}
\|u_h(r)\|^4_\HH dr} \Big[   -(2-3\eta-L_2) \|U(s)\|^2   +C(\eta)
\left(1+
  |h(s)|_{0}^2 \right) |U(s)|^2  \Big].
\end{align}
First consider the case of a general (random) control function $h$.
Below we use the notations
\[
X(t)=
\sup_{0\leq s\leq t}\left\{ e^{- a\int_0^{s\wedge \tau_N}
\|u_h(r)\|^4_\HH dr}
 |U(s\wedge \tau_N)|^2\right\}, ~~
Y(t)=
 \int_0^{t\wedge \tau_N}\!\! \! e^{- a\int_0^{s}
 \|u_h(r)\|^4_\HH dr}
 \|U(s)\|^2 ds.
\]
Then it follows from
\eqref{ito-psi} and \eqref{Psi} that for $3\eta < 2-L_2$,
\[
X(t)+ (2-3\eta-L_2) Y(t)\le C(\eta)
\int_0^t\left(1+ |h(s)|_{0}^2 \right) X(s) ds +I(t),
\]
where $I(t)=\sup_{0\leq s\leq t}\left|\Phi(s\wedge \tau_N)\right|$.
An argument similar to that used to prove \eqref{estimate3}, based on the
Burkholder-Davies-Gundy inequality, ({\bf C2}) and Schwarz's inequality,
yields    that for $t\in[0, T]$ and $\beta>0$,
\begin{align*}
\EX I(t) & \le 6\,  \EX\left[
\int_0^{t\wedge \tau_N}\!\!  e^{- 2a\int_0^{s}
 \|u_h(r)\|^4_\HH dr} |U(s)|^2  |\s(u_{h}(s))-\s(v(s))|^2_{L_Q} ds
 \right]^{1/2}
 \\
 & \le \beta \, \EX X(t)+\frac{9 L_1}{\beta}\int_0^t\EX X(s) ds +
 \frac{9 L_2}{\beta}\, \EX Y(t).
\end{align*}
Now we are in position to apply Lemma~\ref{lemGronwall}.
If we choose $\eta=1/3$, $2\beta=\exp\{-C(1/3)(T+M)\}$, then \eqref{Grw-cond}
holds under the condition $L_2\left(1+36 \exp\{2C(1/3)(T+M)\}\right)\le 1$.
Therefore,
since   $\sup_{0\leq s\leq T}\left\{ e^{- a\int_0^{s\wedge \tau_N}
\|u_h(r)\|^4_\HH dr}
 |U(s\wedge \tau_N)|^2\right\}  \leq2N$, relation \eqref{Gronwall}
  implies  that $\EX X(t)=0$ for all $t$ and hence,
\begin{equation}\label{uniq-s}
  \EX\;  \sup_{0\leq s\leq T} \; \left\{ e^{- a\int_0^{s\wedge \tau_N}
 \|u_h(r)\|^4_\HH dr} \;  |U(s \wedge \tau_N)|^2\right\}  =0.
 \end{equation}
Since
 $\lim_{N\to \infty} \tau_N =T$
 a.s.,   and by \eqref{boundgeneral3}
we have a.s. $\int_0^T \|u_h(s)\|[_{\mathcal H}^4
ds <\infty$,
 we deduce that
 $|U(s, \om)| =0$ a.s. on $\Omega_T$. Thus, we conclude that
$u_h(t)=v(t)$, a.s., for every $t \in [0, T]$ which
yields  the uniqueness statement in Theorem~\ref{th3.1}
for a general control function.
\smallskip
\par
Suppose now that we only have  $L_2<2$ and that $h$ possesses a deterministic bound
$\psi(t)\in L^2(0,T)$;  let $\eta\in ]0, \frac{2-L_2}{3}]$.
Then it follows from
\eqref{ito-psi} and \eqref{Psi} that
\[
V_N(t)\le C(\eta)
\int_0^{t}\Big[ 1+ |\psi(s)|^2 \Big]
 V_N(s)   ds\quad
\mbox{with}\quad V_N(t)=  \EX
 e^{-a \int_0^{t\wedge \tau_N}   \|u_h(r)\|^4_\HH dr}
 |U(t\wedge \tau_N)|^2.
\]
Since the function $s \mapsto     |\psi|^2 $  belongs to
 $L^1(0,T)$, we can apply the Gronwall lemma  to obtain
 \eqref{uniq-s} and  to conclude the proof for the case considered.
\smallskip
\par
Finally suppose  that $\tilde{\sigma}=\sigma$ where $\sigma$ satisfies condition {\bf (C2)} with
  $L_2<2$.
For $h\neq 0$ set $\tilde{W}^h_t=W_t+\int_0^t h(s)\, ds$ and let $\tilde{\PP}$ be the probability
defined on $(\Omega,{\mathcal F_t})$ by
\[ \frac{d\tilde{\PP}}{d\PP}=\exp\Big( -\int_0^t h(s)\, dW_s
-\frac{1}{2} \int_0^t |h(s)|_0^2\, ds\Big).\]
 The Girsanov theorem implies that $\tilde{W}^h$
is a $\tilde{\PP}$ Brownian motion with the same covariance operator $Q$.
The above arguments prove that  under $\tilde{\PP}$,
the evolution equation
\[ u_h(t)=\xi + \int_0^t F(u_h(s))\, ds + \int_0^t \sigma(u_h(s))\, d\tilde{W}^h_s\]
has a unique solution in $X$.
Thus the Girsanov theorem implies that under  ${\PP}$, \eqref{uh} has a unique solution in $X$.
Finally, once well-posedness is proved, computations similar to that used
to obtain \eqref{Galerkin1} in the case $p=2$ can  be used to deduce
that \eqref{eq3.1} holds; this  completes the proof of Theorem~\ref{th3.1}.
\hfill $\Box$
\smallskip\par
Note that it follows from  the consideration above that  we only
need the requirement  \eqref{tilde-s-b}
concerning the growth of $\tilde\s$ in order to obtain
weak compactness of the sequence $ \tilde{\s}_{n}\big(u_{n,h}(s)\big)h(s)$
in $L^p(\Om_T; H)$ for $p=4/3>1$.
This weak compactness makes it possible to pass to the limit
in the term
$\EX  \int_0^T\big( \tilde{\s}_{n}\big(u_{n,h}(s)\big)h(s), v(s)\big) ds$
in the expression for $Z^1_n$ for elements $v$ from the class  ${\mathcal X}$
which contains the limiting function $u_h$.


\vspace{.5cm}

\noindent {\bf Acknowledgements.} This work was partially done
in the fall 2007
while the authors were  visiting the Mittag Leffler
Institute, Sweden,   which provided financial support.
They would like to thank the center for
 excellent working conditions  and
 a very friendly  atmosphere.

\end{document}